\documentclass[aos]{imsart}

\RequirePackage[OT1]{fontenc}
\RequirePackage{amsthm,amsmath}
\usepackage{amssymb, enumerate}
\RequirePackage[numbers]{natbib}
\RequirePackage[colorlinks,citecolor=blue,urlcolor=blue]{hyperref}

\arxiv{1708.07642}

\startlocaldefs
\numberwithin{equation}{section}
\theoremstyle{plain}
\newtheorem{thm}{Theorem}[section]
\newtheorem{remark}{Remark}[section]
\newtheorem{lemma}{Lemma}[section]
\newtheorem{definition}{Definition}[section]
\newtheorem{proposition}{Proposition}[section]
\newtheorem{Corollary}{Corollary}[section]
\newtheorem{assumption}{Assumption}[section]
\endlocaldefs

\def\rank{{\rm rank}}

\def\trace{\mathop{\rm tr}}

\begin{document}

\begin{frontmatter}
\title{Efficient Estimation of Linear Functionals of Principal Components} 
\runtitle{Efficient estimation in PCA}

\begin{aug}
%
%
%
%

\author{\fnms{Vladimir} \snm{Koltchinskii}\thanksref{a}\ead[label=e1]{vlad@math.gatech.edu}},
\author{\fnms{Matthias} \snm{L\"offler}\thanksref{d}
\ead[label=e3]{m.loffler@statslab.cam.ac.uk}}
\and
\author{\fnms{Richard} \snm{Nickl}\thanksref{d}\ead[label=e4]{r.nickl@statslab.cam.ac.uk}}
\thankstext{a}{Supported in part by NSF Grants DMS-1810958, DMS-1509739 and CCF-1523768}
\thankstext{d}{Supported by ERC Grant UQMSI/647812}
\affiliation{Georgia Institute of Technology\thanksmark{a} and University of Cambridge\thanksmark{d}}
\address[a]{School of Mathematics \\
Georgia Institute of Technology Atlanta \\ GA 30332-0160 \\ United States of America. \\
	\printead{e1}}

\address[d]{Statistical Laboratory \\
	University of Cambridge\\
	CB3 0WB Cambridge \\
	United Kingdom \\
	\printead{e3,e4}}
\end{aug}

\begin{abstract}
We study principal component analysis (PCA) for mean zero i.i.d. Gaussian observations $X_1,\dots, X_n$ in a separable Hilbert space $\mathbb{H}$ with unknown covariance operator 
$\Sigma.$ 
The complexity of the problem is characterized by its effective rank 
${\bf r}(\Sigma):= \frac{{\rm tr}(\Sigma)}{\|\Sigma\|},$ where 
${\rm tr}(\Sigma)$ denotes the trace of $\Sigma$ and $\|\Sigma\|$
denotes its operator norm. 
We develop a method of bias reduction 
in the problem of estimation of linear functionals of eigenvectors of $\Sigma.$
Under the assumption that ${\bf r}(\Sigma)=o(n),$ we establish the asymptotic normality and asymptotic properties of the risk of the 
resulting estimators and prove matching minimax lower bounds, showing their semi-parametric optimality.


\end{abstract}



\end{frontmatter}

\section{Introduction}
Principal Component Analysis (PCA) is commonly used as a dimension reduction technique for high-dimensional data sets. Assuming a general framework where the data lies in a Hilbert space $\mathbb{H},$ PCA can be applied to a wide range of problems such as functional data analysis  \cite{RamsaySilverman, LilaAstonSangalli} or machine learning \cite{BlanchardBousquetZwald}. \\
The parametric setting has been well understood since the 1960's (e.g. \cite{Anderson1963} and \cite{DauxoisPousseRomain1982}) and the asymptotic distribution of sample eigenvalues and sample eigenvectors is well known. For high-dimensional data, where the dimension $p=p(n) \rightarrow \infty$ with the sample size $n,$  
the spiked covariance model introduced by Johnstone in  \cite{JohnstoneAos2001}
has been the most common framework to study the asymptotic properties of principal 
components. In this model, it is assumed that the covariance matrix is given by a 'spike' and a noise part, that is 
$$\Sigma=\sum_{j=1}^l s_i (\theta_i \otimes \theta_i) + \sigma^2 I_p,$$
where $\sum_{j=1}^l s_i (\theta_i \otimes \theta_i)$ is a low rank 
covariance matrix involving several orthonormal components ('spikes') $\theta_i$
and $\sigma^2 I_p$ is the covariance of the noise. 
Error bounds in this model, based on perturbation analysis, were studied in \cite{Nadler2008}. Moreover, if $\frac{p}{n} \rightarrow c \in (0,1]$ the asymptotic distribution of sample eigenvectors was derived in \cite{PaulStatisticaSinica} and in more general asymptotic regimes in \cite{WangFan}. Assuming sparsity of the eigenvectors (sparse PCA), inference is possible even when $\frac{p}{n} \rightarrow \infty$. This model has recently received substantial attention, e.g. \cite{CaiMaWu2013, BerthetRiggolet, VuLei2013, WangBerthetSamworth, GaoZhou}. 

\smallskip

More recently, a so-called 'effective rank' setting for PCA has been considered, for example, in \cite{KoltchinskiiLouniciPCAAHP, KoltchinskiiLouniciPCABernoulli, KoltchinskiiLouniciPCAAOS, Vershynin, ReissWahl, NaumovSpokoinyUlyanov}. 
In this dimension-free setting, it is assumed that the covariance 
$\Sigma$ is an operator acting in a Hilbert space $\mathbb{H},$
no structural assumptions are made about $\Sigma$ and its 'complexity' 
is characterized by the \emph{effective rank} 
${\bf r}(\Sigma):=\trace(\Sigma)/\|\Sigma\|,$
$\trace(\Sigma)$ denoting the trace and $\|\Sigma\|$ denoting 
the operator (spectral) norm of $\Sigma.$
In a series of papers \cite{KoltchinskiiLouniciPCABernoulli, KoltchinskiiLouniciPCAAHP,  KoltchinskiiLouniciPCAAOS, KoltchinskiiLouniciPCAarxiv},  Koltchinskii and Lounici derived sharp bounds on the spectral norm loss of estimation of $\Sigma$ 
by the sample covariance $\hat \Sigma$ that provide complete characterization 
of the size of $\|\hat \Sigma-\Sigma\|$
in terms of $\|\Sigma\|$ and ${\bf r}(\Sigma),$ and obtained error bounds and limiting results for empirical spectral projection operators and eigenvectors of $\hat \Sigma$
under the assumption that ${\bf r}(\Sigma)=o(n)$ as $n\to\infty.$
In a recent paper \cite{NaumovSpokoinyUlyanov}, Naumov et. al. constructed bootstrap confidence sets for spectral projections in a lower dimensional regime where ${\bf r}(\Sigma)=o(n^{1/3})$. In \cite{ReissWahl}, Reiss and Wahl considered the reconstruction error for spectral projections. \\ 
In this paper, we further develop the results of \cite{KoltchinskiiLouniciPCAAHP} and \cite{KoltchinskiiLouniciPCAAOS} in the direction of semi-parametric statistics. In particular, we develop  a bias reduction method in the problem of 
estimation of linear functionals of principal components (eigenvectors of $\Sigma$)
and show asymptotic normality of the resulting de-biased estimators 
under the assumption that ${\bf r}(\Sigma)=o(n).$ We prove a non-asymptotic risk 
lower bound that asymptotically exactly matches our upper bounds, thus establishing rigorously the semi-parametric optimality of our estimator in a general dimension-free setting (as long as ${\bf r}(\Sigma)=o(n)$). 

The problem of $\sqrt n$-consistent estimation of low-dimensional functionals of high-dimensional parameters has received increased attention in recent years, and in various models semi-parametric efficiency of regularisation-based estimators has been studied, see for instance \cite{vdgBuhlmannRitovDezeureAOS, Jankovavdg, RenSunZhangZhou, NingLiu, FanRigolletWang}. Moreover, the paper \cite{GaoZhou} develops Bernstein-von-Mises (BvM) results for functionals of covariance matrices in situations where bias is asymptotically negligible. 
 While formal calculations of the Fisher information in such models indicate optimality of these procedures, a rigorous interpretation of such efficiency claims requires some care: the standard asymptotic setting for semi-parametric efficiency \cite{vdvAsymptoticStatistics} can not be straightforwardly applied because parameters in high-dimensional models are not fixed but vary with sample size $n$, so that establishing LAN expansions to apply Le Cam theory is not always possible or even desirable. In \cite{Jankovavdg} some non-asymptotic techniques have been suggested under conditions that ensure asymptotic negligibility of the bias of candidate estimators. We take here a different approach, based on using the van Trees' inequality \cite{GillLevit} to construct non-asymptotic lower bounds for the minimax risk in our estimation problem that match the upper bound \textit{exactly} in the large sample limit.

\section{Preliminaries}

\subsection{Some notations and conventions.}
Let $\mathbb{H}$ be a separable Hilbert space. In what follows, 
$\langle \cdot, \cdot \rangle$ denotes the inner product of ${\mathbb H}$
and also, with a little abuse of notation, the Hilbert--Schmidt inner product 
between Hilbert--Schmidt operators acting on ${\mathbb H}.$ Similarly, the notation $\|\cdot\|$ is used both for the norm of vectors in ${\mathbb H}$ and for the operator (spectral) norm of bounded linear operators in ${\mathbb H}.$
For a nuclear operator $A,$ $\trace(A)$ denotes its trace. 
We use the notation $\|\cdot\|_p,$ $1\leq p\leq \infty$ for 
the Schatten $p$-norms of operators in ${\mathbb H}:$ 
$\|A\|_p:= (\trace(|A|^p))^{1/p},$ where $|A|=\sqrt{A^{\ast} A},$
$A^{\ast}$ being the adjoint operator of $A.$
For $p=1,$ $\|A\|_1$ is the nuclear norm; for $p=2,$ $\|A\|_2$
is the Hilbert--Schmidt norm; for $p=\infty,$ $\|A\|_{\infty}=\|A\|$
is the operator norm.

Given vectors $u, v \in \mathbb{H},~ u \otimes v$ denotes the tensor product of $u$ and $v:$ 
$$
(u\otimes v):{\mathbb H}\mapsto {\mathbb H}, 
(u \otimes v)w:=\langle v, w \rangle u.
$$
Given bounded linear operators $A,B:{\mathbb H}\mapsto {\mathbb H},$
$A\otimes B$ denotes their tensor product:
$$
(A\otimes B)(u\otimes v)= Au\otimes Bv,\ u,v\in {\mathbb H}.
$$
Note that $A\otimes B$ can be extended (by linearity and continuity)
to a bounded operator in the Hilbert space ${\mathbb H}\otimes {\mathbb H},$
which could be identified with the space of Hilbert--Schmidt operators in 
${\mathbb H}.$ It is easy to see that, for a Hilbert--Schmidt operator $C,$ we have $(A\otimes B)C= ACB^{\ast}$ (in the finite-dimensional case, this defines the so called 
Kronecker product of matrices). On a couple of occasions, we might need 
to use the tensor product of Hilbert--Schmidt operators $A,B,$ viewed as vectors 
in the space of Hilbert--Schmidt operators. For this tensor product, we use 
the notation $A\otimes_v B.$

Throughout the paper, the following notations will be used: for nonnegative $a,b,$ $a \lesssim b$ means that there exists a numerical constant $c>0$ such that $a \leq cb;$
$a\gtrsim b$ is equivalent to $b\lesssim a;$ finally, $a\asymp b$ is equivalent 
to $a\lesssim b$ and $b\lesssim a.$ Sometimes, constant $c$ in the above relationships could depend on some parameter $\gamma.$ In this case, we 
provide signs $\lesssim,$ $\gtrsim$ and $\asymp$ with subscript $\gamma.$
For instance, $a\lesssim_{\gamma} b$ means that there exists a constant 
$c_{\gamma}>0$ such that $a\leq c_{\gamma}b.$

In many places in the proofs, we use exponential bounds for some random
variables, say, $\xi$ of the following 
form: for all $t\geq 1$ with probability at least $1-e^{-t},$
$\xi \leq Ct.$ In some cases, it would follow from our arguments 
that the inequality holds with a slightly different probability, say, 
at least $1-3e^{-t}.$ In such cases, it is easy to rewrite the bound 
again as $1-e^{-t}$ by adjusting the value of constant $C.$ Indeed,
for $t\geq 1$ with probability at least $1-e^{-t} = 1- 3e^{-t-\log(3)},$
we have $\xi \leq C(t+\log (3))\leq 2\log(3) C t.$ We will use 
such an adjustment of the constants in many proofs, often,
without further notice. 

\subsection{Bounds on sample covariance.}
Let 
$X$ be a Gaussian vector in ${\mathbb H}$
with mean ${\mathbb E}X=0$ and covariance operator $\Sigma:={\mathbb E}(X\otimes X).$ Given i.i.d. observations $X_1,\dots, X_n$ of $X,$ let $\hat \Sigma= \hat \Sigma_n$
be the sample (empirical) covariance operator defined as follows:
$$
\hat \Sigma := n^{-1}\sum_{j=1}^n X_j\otimes X_j.
$$

\begin{definition}
The effective rank of the covariance operator $\Sigma$
is defined as 
$${\bf r}(\Sigma):=\frac{\trace (\Sigma)}{\|\Sigma\|}.$$ 
\end{definition}

The role of the effective rank as a complexity parameter in covariance 
estimation is clear from the following result proved in \cite{KoltchinskiiLouniciPCABernoulli}.

\begin{thm}
\label{KL-Bernoulli}
Let $X$ be a mean zero Gaussian random vector in ${\mathbb H}$ with covariance operator $\Sigma$ and let $\hat \Sigma$ be the sample covariance based on 
i.i.d. observations $X_1,\dots, X_n$ of $X.$ Then
\begin{equation}
\label{expect_hat_Sigma}
{\mathbb E}\|\hat \Sigma-\Sigma\| \asymp 
\|\Sigma\|
\biggl(\sqrt{\frac{{\bf r}(\Sigma)}{n}} \bigvee \frac{{\bf r}(\Sigma)}{n}\biggr).
\end{equation}
\end{thm}

This result shows that the size of the properly rescaled operator norm deviation 
of $\hat \Sigma$ from $\Sigma,$ $\frac{{\mathbb E}\|\hat \Sigma-\Sigma\|}{\|\Sigma\|},$
is characterized up to numerical constants by the ratio $\frac{{\bf r}(\Sigma)}{n}.$
In particular, the condition ${\bf r}(\Sigma)=o(n)$ is necessary and sufficient 
for operator norm consistency of $\hat \Sigma$ as an estimator of $\Sigma.$
In addition to this, the following concentration inequality for  $\|\hat \Sigma-\Sigma\|$
around its expectation was also proved in \cite{KoltchinskiiLouniciPCABernoulli}.

\begin{thm}
Under the conditions of the previous theorem, for all $t\geq 1$ with probability 
at least $1-e^{-t}$
\begin{equation}
\label{conc_hat_Sigma}
\Bigl|\|\hat \Sigma-\Sigma\|-{\mathbb E}\|\hat \Sigma-\Sigma\|\Bigr| \lesssim 
\|\Sigma\| \biggl(\biggl(\sqrt{\frac{{\bf r}(\Sigma)}{n}} \bigvee 1\biggr)\sqrt{\frac{t}{n}}
\bigvee \frac{t}{n}\biggr).
\end{equation}
\end{thm}

It immediately follows from the bounds \eqref{expect_hat_Sigma} and \eqref{conc_hat_Sigma} that, for all $t\geq 1$ with probability at least $1-e^{-t}$
\begin{equation}
\label{exp_bd_Sigma}
\|\hat \Sigma-\Sigma\|\lesssim \|\Sigma\|
\biggl(\sqrt{\frac{{\bf r}(\Sigma)}{n}} \bigvee \frac{{\bf r}(\Sigma)}{n}
\bigvee \sqrt{\frac{t}{n}}\bigvee \frac{t}{n}\biggr)
\end{equation}
and, for all $p\in [1,\infty),$ 
\begin{equation}
\label{hatSigmap}
{\mathbb E}^{1/p}\|\hat \Sigma-\Sigma\|^p \lesssim_p 
\|\Sigma\|
\biggl(\sqrt{\frac{{\bf r}(\Sigma)}{n}} \bigvee \frac{{\bf r}(\Sigma)}{n}\biggr).
\end{equation}

\subsection{Perturbation theory and empirical spectral projections.}

The covariance operator $\Sigma$ is self-adjoint, positively semidefinite 
and nuclear. It has spectral decomposition 
$$
\Sigma=\sum_{r\geq 1} \mu_r P_r,
$$ 
where $\mu_r$ are distinct strictly positive eigenvalues of $\Sigma$
arranged in decreasing order 
and $P_r$ are the corresponding spectral projection operators.
For $r\geq 1,$ $P_r$ is an orthogonal projection on the eigenspace 
of the eigenvalue $\mu_r.$ The dimension of this eigenspace is finite and 
will be denoted by $m_r.$
The eigenspaces corresponding to different eigenvalues $\mu_r$ are 
mutually orthogonal. Denote by $\sigma(\Sigma)$ the spectrum of operator 
$\Sigma$ and let $\lambda_j=\lambda_j(\Sigma), j\geq 1$ be the eigenvalues of 
$\Sigma$ arranged in a non-increasing order and repeated with their multiplicities. Denote $\Delta_r:=\{j: \lambda_j =\mu_r\}, r\geq 1.$ Then ${\rm card}(\Delta_r)=m_r.$
The $r$-th spectral gap is defined as 
$$g_r= g_r(\Sigma):={\rm dist}(\mu_r; \sigma(\Sigma)\setminus \{\mu_r\}).$$ 
Let $\bar g_r=\bar g_r(\Sigma):= \min_{1\leq s\leq r} g_s.$

We turn now to the definition of empirical spectral projections of sample covariance 
$\hat \Sigma$ that could be viewed as estimators of the true spectral projections $P_r, r\geq 1.$ In \cite{KoltchinskiiLouniciPCAAHP}, the following definition was used:
let $\hat P_r$ be the orthogonal projection on the direct sum of eigenspaces 
of $\hat \Sigma$ corresponding to its eigenvalues 
$\{\lambda_j(\hat \Sigma): j\in \Delta_r\}.$ This is not a perfect 
definition of a statistical estimator since the set $\Delta_r$ is 
unknown and it has to be recovered from the spectrum $\sigma(\hat \Sigma)$
of $\hat \Sigma.$

When $\hat \Sigma$ is close to $\Sigma$ in the operator norm, the spectrum 
$\sigma(\hat \Sigma)$ of $\hat \Sigma$ is a small perturbation of the spectrum 
$\sigma(\Sigma)$ of $\Sigma.$ This could be quantified by the following inequality
that goes back to H. Weyl:
\begin{equation}
\label{H. Weyl}
\sup_{j\geq 1}|\lambda_j(\hat \Sigma)-\lambda_j(\Sigma)|
\leq \|\hat \Sigma-\Sigma\|.
\end{equation}
It easily follows from this inequality that, if $\|\hat \Sigma-\Sigma\|$
is sufficiently small, then the eigenvalues $\lambda_j(\hat \Sigma)$
of $\hat \Sigma$ form well separated clusters around the eigenvalues 
$\mu_1, \mu_2, \dots$ of $\Sigma.$ To make the last claim more 
precise, consider a finite or countable bounded set $A\subset {\mathbb R}_+$
such that $0\in A$ and $0$ is the only limit point (if any) of $A.$ 
Given $\delta>0,$ define 
$
\lambda_{\delta}(A):= \max\bigl\{\lambda \in A: (\lambda-\delta,\lambda)\cap A=\emptyset\bigr\}
$
and let $T_{\delta}(A) := A\setminus [0,\lambda_{\delta}(A)).$
The set $T_{\delta}(A)$ will be called the top $\delta$-cluster of $A.$
Let $A_1:=T_{\delta}(A), A_2:=T_{\delta}(A\setminus A_1), 
A_3:= T_{\delta}(A\setminus (A_1\cup A_2)), \dots$ and 
$\nu=\nu_{\delta}:= \min\{j: A_{j+1}=\emptyset\}.$ Obviously, $\nu<\infty.$
We will call the sets $A_1, \dots, A_{\nu}$ the $\delta$-clusters of $A.$
They provide a partition of $A$ into sets separated by the gaps of length 
at least $\delta$ and such that the gaps between the points inside each 
of the clusters are smaller than $\delta.$ 

The next lemma easily follows from inequality \eqref{H. Weyl}.

\begin{lemma}
\label{delta_clusters}
Let $\delta>0$ be such that, for some $r\geq 1,$ 
$$
\|\hat \Sigma-\Sigma\|< \delta/2\ \ {\rm and}\ \ \delta<\frac{\bar g_r}{2}.
$$
Let $\hat A_1^{\delta}, \dots, \hat A_{\nu}^{\delta}$ be the $\delta$-clusters of the set $\sigma(\hat \Sigma).$ Then $\nu \geq r$ and, for all $1\leq s\leq r$
$$
\hat A_s^{\delta} \subset (\mu_s-\delta/2, \mu_s+\delta/2)\ \ {\rm and}\ \
\{j: \lambda_j(\hat \Sigma)\in \hat A_s^{\delta}\}=\Delta_s.
$$
\end{lemma}
 
Given $\delta>0$ and $\delta$-clusters $\hat A_1^{\delta}, \dots, \hat A_{\nu}^{\delta}$ of $\sigma(\hat \Sigma),$ define, for $1\leq s\leq \nu,$ the empirical spectral projection $\hat P_s^{\delta}$ 
as the orthogonal projection on the direct sum of eigenspaces of $\hat \Sigma$
corresponding to its eigenvalues from the cluster $\hat A_s^{\delta}.$ 
It immediately follows from Lemma \ref{delta_clusters} that, under its assumptions
on $\delta,$ $\hat P_s^{\delta}= \hat P_s, s=1,\dots, r.$ 

In the following sections, we will be interested in the problem of estimation of 
spectral projections in the case when the true covariance $\Sigma$ belongs 
to certain subsets of the following class of covariance operators:
$$
{\mathcal S}^{(r)}({\frak r};a):=\bigl\{\Sigma: {\bf r}(\Sigma)\leq {\frak r}, \frac{\|\Sigma\|}{\bar g_r(\Sigma)}\leq a\bigr\}, 
$$
where $a>1, {\frak r}>1.$ We will allow the effective rank to be large,  
${\frak r}={\frak r}_n\to \infty,$ but not too large such that ${\frak r}_n=o(n)$
as $n\to \infty.$ For $\Sigma\in {\mathcal S}^{(r)}({\frak r};a),$
we take $\delta:= \tau \|\hat \Sigma\|$ for a sufficiently small 
value of the constant $\tau>0$ in the definition of spectral projections
$\hat P_s^{\delta}.$ 

The following lemma is an easy consequence of the exponential bound \eqref{exp_bd_Sigma}.

\begin{lemma}
\label{hatPdelta}
Suppose $a>1$ and ${\frak r}_n=o(n)$ as $n\to \infty.$
Take $\tau \in \bigl(0, \frac{1}{4a}\wedge 2\bigl)$ and 
$\delta:=\tau \|\hat \Sigma\|.$ Then, there exists a numerical constant $\beta>0$
such that, for all large enough $n,$ 
$$
\sup_{\Sigma\in {\mathcal S}^{(r)}({\frak r};a)}{\mathbb P}_{\Sigma}
\{\exists s=1,\dots, r: \hat P_s^{\delta}\neq \hat P_s\}\leq e^{-\beta \tau^2 n}.
$$
\end{lemma}
\begin{proof}
By \eqref{exp_bd_Sigma} with $t:= \beta \tau^2 n,$ we obtain that
$$
\sup_{\Sigma \in {\mathcal S}^{(r)}({\frak r};a)}
{\mathbb P}_{\Sigma}\bigl\{\|\hat \Sigma-\Sigma\|\geq C\|\Sigma\|
\bigg(\sqrt{\frac{{\frak r_n}}{n}}\bigvee \sqrt{\frac{\beta \tau^2 n}{n}}\biggr)\bigr\}\leq e^{-\beta \tau^2 n},
$$
where $C>0$ is a numerical constant. Take $\beta = \frac{1}{16C^2}$ and 
note that, for all large enough $n,$ $C\sqrt{\frac{{\frak r_n}}{n}}\leq \tau/4$
to obtain that 
$$
\sup_{\Sigma \in {\mathcal S}^{(r)}({\frak r};a)}
{\mathbb P}_{\Sigma}\{\|\hat \Sigma-\Sigma\|\geq (\tau/4) \|\Sigma\|\}
\leq e^{-\beta \tau^2 n},
$$
Since $\tau/4 \leq 1/2,$ we easily obtain that, for all $\Sigma \in {\mathcal S}^{(r)}({\frak r};a)$
and for all $n$ large enough with probability at least $1-e^{-\beta \tau^2 n},$ 
$(1/2)\|\Sigma\|\leq \|\hat \Sigma\|\leq 2\|\Sigma\|.$
This implies that with the same probability (and on the same event)
$$
\|\hat \Sigma-\Sigma\|< (\tau/4)\|\Sigma\|\leq (\tau/2) \|\hat \Sigma\|=\delta/2.
$$
On the other hand, for all $\Sigma \in {\mathcal S}^{(r)}({\frak r};a),$
$$
\delta = \tau \|\hat \Sigma\|\leq 2\tau \|\Sigma\|
< \frac{1}{2a} \|\Sigma\| \leq \frac{\bar g_r(\Sigma)}{2}.
$$
It remains to use Lemma \ref{delta_clusters} to complete
the proof.
\end{proof} 
 
In the proofs of the main results of the paper, we deal for the most part with spectral projections 
$\hat P_r$ that were studied in detail in \cite{KoltchinskiiLouniciPCAAHP}. We  
use Lemma \ref{hatPdelta} to reduce the results for $\hat P_r^{\delta}$ to the results for $\hat P_r.$ 

\section{Main Results}

Our main goal is to develop an efficient estimator of the linear functional 
$\langle \theta_r, u\rangle,$ where $u\in {\mathbb H}$ is a given vector and $\theta_r=\theta_r(\Sigma)$ is a unit eigenvector of the  unknown covariance operator $\Sigma$ corresponding to its $r$--th eigenvalue $\mu_r,$ which is assumed to be simple (that is, of multiplicity $m_r=1$). The corresponding spectral projection $P_r$ is one-dimensional:
$P_r=\theta_r \otimes \theta_r.$  A ``naive" plug-in estimator of $P_r$ is the empirical 
spectral projection $\hat P_r^{\delta}$ with $\delta= \tau\|\hat \Sigma\|$ for a suitable 
choice of a small constant $\tau,$ as described in Lemma \ref{hatPdelta}. According 
to this lemma and under its assumptions, $\hat P_r^{\delta}$ coincides with a high
probability with the one-dimensional empirical spectral projection $\hat P_r:=\hat \theta_r\otimes \hat \theta_r,$ where $\hat \theta_r$ is the corresponding unit eigenvector 
of $\hat \Sigma.$ As an estimator of $\theta_r,$ 
we can use an arbitrary unit vector $\hat \theta_r^{\delta}$
from the eigenspace ${\rm Im}(\hat P_r^{\delta}),$ which with a high-probability 
coincides with $\pm \hat \theta_r$ (under conditions of Lemma \ref{hatPdelta}). 
In case $r=1,$ when 
the top eigenvalue $\mu_1=\|\Sigma\|$ of $\Sigma$ is simple and the goal 
is to estimate a linear functional of the top principal component $\theta_1,$
there is no need to use $\delta$-clusters to define an estimator of $\theta_1$ since $\hat \theta_1$
(a unit eigenvector in the eigenspace of the top eigenvalue $\|\hat \Sigma\|$ of $\hat \Sigma$) is already a legitimate estimator.  

Note that both $\theta_r$
and $-\theta_r$ are unit eigenvectors of $\Sigma,$ so, strictly speaking, $\langle \theta_r,u\rangle$ can be estimated only up to its sign. 
In what follows, we assume that $\hat \theta_r^{\delta}$ and $\theta_r$
(or, whenever is needed, $\hat \theta_r$ and $\theta_r$) are \emph{properly aligned} in the sense 
that 
$\langle \hat \theta_r^{\delta}, \theta_r\rangle \geq 0$ (which is always the case either for $\theta_r,$ or for $-\theta_r$). This allows
us to view $\langle \hat \theta_r^{\delta}, u\rangle$ as an estimator of $\langle \theta_r,u\rangle.$

It was shown in \cite{KoltchinskiiLouniciPCAAHP} that ``naive" plug-in estimators 
of the functional $\langle \theta_r,u\rangle,$ such as $\langle \hat \theta_r^{\delta},u\rangle$ or $\langle \hat \theta_r,u\rangle,$ are biased with the bias becoming substantial enough to affect the efficiency 
of the estimator or even its convergence rates as soon as the effective rank 
is large enough, namely, ${\bf r}(\Sigma)\gtrsim n^{1/2}.$ Moreover, 
it was shown that the quantity 
$$
b_r = b_r(\Sigma):={\mathbb E}_{\Sigma}\langle \hat \theta_r, \theta_r\rangle^2 -1 \in [-1,0]
$$
plays the role of a bias parameter. In particular, the results of \cite{KoltchinskiiLouniciPCAAHP} imply that the random variable 
$\langle \hat \theta_r, u\rangle$ concentrates around $\sqrt{1+b_r}\langle \theta_r,u\rangle$ (rather than 
around $\langle \theta_r,u\rangle$) with the size of the deviations of order $O(n^{-1/2})$
provided that ${\bf r}(\Sigma)=o(n)$ as $n\to \infty.$ Thus, the bias of $\langle \hat \theta_r, u\rangle$ as an estimator of $\langle \theta_r,u\rangle$ is of the order 
$(\sqrt{1+b_r}-1) \langle \theta_r,u\rangle \asymp b_r \langle \theta_r,u\rangle.$
It was shown in \cite{KoltchinskiiLouniciPCAAHP} that $|b_r| \lesssim \frac{{\bf r}(\Sigma)}{n}$ and it will be proved below in this paper that, in fact, $|b_r| \asymp \frac{{\bf r}(\Sigma)}{n}$
(see Lemma \ref{brbou} and bounds \eqref{bd_Ar_up}, \eqref{bd_Ar_down}). 
This fact implies that, indeed, the bias of $\langle \hat \theta_r,u\rangle$
(and of  $\langle \hat \theta_r^{\delta},u\rangle$) is not negligible and affects the convergence rate as soon as $\frac{{\bf r}(\Sigma)}{n^{1/2}}\to \infty.$
This resembles the situation in sparse regression (see e.g. \cite{JavanmardMontanari, vdgBuhlmannRitovDezeureAOS, ZhangZhang}): If $p$ denotes the dimension of the model and $s$ its sparsity and if $s \log(p)=o(n^{1/2}),$ the bias of a de-sparsified LASSO estimator for the regressor $\beta$ is negligible, which makes it possible to prove asymptotic normality of linear forms of $\beta$. 
On the other hand, if $s \log(p) \gg n^{1/2},$ Cai and Guo \cite{CaiGuo} proved that adaptive confidence sets for linear forms  do not exist in general. This implies that any attempt to further de-bias the de-sparsified LASSO or any other estimator to prove asymptotic normality is deemed to fail. 
Contrary to this, in our case estimation of the bias parameter $b_r$ is possible
(as will be shown below).

We will state 
a uniform (and somewhat stronger) version of some of the results
of \cite{KoltchinskiiLouniciPCAAHP} on asymptotic normality of linear forms 
$$
\sqrt{n}(\langle \hat \theta_r^{\delta},u\rangle- \sqrt{1+b_r(\Sigma)}\langle\theta_r (\Sigma),u\rangle), u\in {\mathbb H}
$$
under the assumption that ${\bf r}(\Sigma) =o(n).$ 
To this end, define the following operator
$$
C_r := \sum_{s\neq r} \frac{1}{\mu_r-\mu_s}P_s,
$$
which is bounded with $\|C_r\|=\frac{1}{g_r}.$
Denote 
$$
\sigma_r^2 (\Sigma;u):= \langle \Sigma \theta_r, \theta_r\rangle \langle \Sigma C_r u, C_r u\rangle
=\mu_r \langle \Sigma C_r u, C_ru\rangle.
$$
Clearly,
\begin{equation}
\label{bd_sigma_r}
\sigma_r^2 (\Sigma;u)\leq \frac{\|\Sigma\|^2}{g_r^2} \|u\|^2.
\end{equation}
Note that, if ${\mathbb H}$ is finite-dimensional (with a fixed dimension) and $\Sigma$ is non-singular, 
then the Fisher information for the model $X\sim N(0;\Sigma)$ is ${\mathbb  I}(\Sigma)
=\frac{1}{2}(\Sigma^{-1}\otimes \Sigma^{-1})$ (see, e.g., \cite{Eaton}). The maximum likelihood estimator $\hat \Sigma$ based on $n$ i.i.d. observations of $X$ (the sample covariance) is then asymptotically normal with $\sqrt{n}$-rate and limit covariance 
${\mathbb I}(\Sigma)^{-1}=2(\Sigma \otimes \Sigma).$ An application of the Delta Method to the smooth function 
$g(\Sigma):= \langle \theta_r(\Sigma), u\rangle$ shows that $g(\hat \Sigma)$ is 
also asymptotically normal with limiting variance 
$\bigl\langle ({\mathbb I}(\Sigma)^{-1}g'(\Sigma), g'(\Sigma)\bigr\rangle,$
which turns out to be equal to $\sigma_r^2(\Sigma;u).$

For $u\in {\mathbb H},$ ${\frak r}>1,$ $a>1$ and $\sigma_0>0,$ consider the following class of covariance operators in ${\mathbb H}:$
$$
{\mathcal S}^{(r)}({\frak r}, a, \sigma_0, u)
:=\bigl\{\Sigma: {\bf r}(\Sigma)\leq {\frak r}, \frac{\|\Sigma\|}{\bar g_r(\Sigma)}\leq a, \sigma_r^2(\Sigma;u)\geq \sigma_0^2\bigr\}.
$$
We emphasize here that we regard $a$ and $\sigma_0$ as fixed constants, but ${\frak r}, \| \Sigma \|$ and $\bar g_r$ may all possibly depend on $n$. For example, this allows that $\| \Sigma \| \rightarrow \infty$ as long as $\bar g_r \rightarrow \infty$ at the same rate as it is the case in factor models as considered in \cite{WangFan}.
Note that some additional conditions on 
${\frak r}, a, \sigma_0, u$ are needed for the class ${\mathcal S}^{(r)}({\frak r}, a, \sigma_0, u)$ to be nonempty. Say, bound \eqref{bd_sigma_r} implies that it is necessary
for this that $\sigma_0^2\leq a^2\|u\|^2.$ It is also obvious that there should be $a>r$
(since $\|\Sigma\|\geq r g_r(\Sigma)$).

We will also need the following assumption on the loss function $\ell.$
\begin{assumption}
\label{assump_loss}
Let $\ell : {\mathbb R}\mapsto {\mathbb R}_+$ be a loss function satisfying the following conditions: 
$\ell(0)=0,$ $\ell(u)=\ell(-u), u\in {\mathbb R},$ $\ell$ is nondecreasing and convex on ${\mathbb R}_+$
and, for some constants $c_1,c_2>0$
$$
\ell (u)\leq c_1 e^{c_2 u}, u\geq 0.
$$
\end{assumption}

The proofs to all our theorems are in fact non-asymptotic and often can be expressed by Berry-Esseen type bounds. However, for a more concise presentation we present  asymptotic statements.

In what follows, $Z$ denotes a standard Gaussian random variable
and $\Phi$ denotes its distribution function.

\begin{thm}
\label{norm_approx}
Let $u\in {\mathbb H},$ $a>1$ and $\sigma_0>0.$ Suppose that ${\frak r}_n >1$ and ${\frak r}_n =o(n)$ as $n\to\infty.$ 
Let $\delta= \tau\|\hat \Sigma\|$ for some $\tau\in 
\bigl(0, \frac{1}{4a}\wedge 2\bigl).$ 
Then 
\begin{eqnarray}
\label{norm_approx_main}
&
\nonumber
\sup_{\Sigma \in {\mathcal S}^{(r)}({\frak r}_n, a, \sigma_0, u)}\sup_{x\in {\mathbb R}}\bigl|{\mathbb P}_{\Sigma}\bigl\{\frac{\sqrt{n}(\langle \hat \theta_r^{\delta},u\rangle- \sqrt{1+b_r(\Sigma)}\langle \theta_r(\Sigma),u\rangle)}{\sigma_r(\Sigma;u)}\leq x\bigr\}-\Phi(x)\bigr| \to 0\ {\rm as}\ n\to\infty.
\end{eqnarray}
Moreover, under Assumption \ref{assump_loss},
\begin{equation}
\nonumber
\sup_{\Sigma \in {\mathcal S}^{(r)}({\frak r}_n, a, \sigma_0, u)}\bigl|{\mathbb E}_{\Sigma}\ell\biggl(\frac{\sqrt{n}(\langle \hat \theta_r^{\delta},u\rangle- \sqrt{1+b_r(\Sigma)}\langle\theta_r(\Sigma),u\rangle)}{\sigma_r(\Sigma;u)}\biggr)-{\mathbb E}\ell(Z)\bigr| \to 0\ {\rm as}\ n\to\infty.
\end{equation}
\end{thm}

The proof of this theorem will be given in Section \ref{sec:norm_approx}
that also includes a number of auxiliary statements used in the proofs of our main 
results on efficient estimation of linear functionals. 

\begin{Corollary}
\label{norm_approx_cor}
Let $u\in {\mathbb H},$ $a>1$ and $\sigma_0>0.$ Suppose that ${\frak r}_n >1$ and ${\frak r}_n =o(\sqrt{n})$ as $n\to\infty.$ 
Let $\delta= \tau\|\hat \Sigma\|$ for some $\tau\in 
\bigl(0, \frac{1}{4a}\wedge 2\bigl).$ 
Then 
\begin{eqnarray}
\label{norm_approx_main_cor}
&
\nonumber
\sup_{\Sigma \in {\mathcal S}^{(r)}({\frak r}_n, a, \sigma_0, u)}\sup_{x\in {\mathbb R}}\bigl|{\mathbb P}_{\Sigma}\bigl\{\frac{\sqrt{n}(\langle \hat \theta_r^{\delta},u\rangle- \langle \theta_r(\Sigma),u\rangle)}{\sigma_r(\Sigma;u)}\leq x\bigr\}-\Phi(x)\bigr| \to 0\ {\rm as}\ n\to\infty.
\end{eqnarray}
Moreover, under Assumption \ref{assump_loss},
\begin{equation}
\nonumber
\sup_{\Sigma \in {\mathcal S}^{(r)}({\frak r}_n, a, \sigma_0, u)}\bigl|{\mathbb E}_{\Sigma}\ell\biggl(\frac{\sqrt{n}(\langle \hat \theta_r^{\delta},u\rangle- \langle\theta_r(\Sigma),u\rangle)}{\sigma_r(\Sigma;u)}\biggr)-{\mathbb E}\ell(Z)\bigr| \to 0\ {\rm as}\ n\to\infty.
\end{equation}
\end{Corollary}

Our next goal is to provide a minimax lower bound 
on the risk of an arbitrary estimator of the linear functional 
$\langle\theta_r(\Sigma),u\rangle$ 
in the case of quadratic loss $\ell(t)=t^2, t\in 
{\mathbb R}.$ 
The proof is based on van Trees' inequality and will be given 
in Section \ref{sec:van Trees}.
Define
$$
\mathring{{\mathcal S}}^{(r)}({\frak r}, a, \sigma_0, u)
:=\bigl\{\Sigma: {\bf r}(\Sigma)<{\frak r}, \frac{\|\Sigma\|}{\bar g_r(\Sigma)}<a, \sigma_r^2(\Sigma;u)> \sigma_0^2\bigr\}, {\frak r}>1, a>1, \sigma_0^2>0,
$$
the interior of the set ${\mathcal S}^{(r)}({\frak r}, a, \sigma_0, u).$

\begin{thm} 
\label{Thm Minimax lower bounds 1}
Let ${\frak r}>1,$ $a>1$ and $\sigma_0>0.$
Suppose $\mathring{{\mathcal S}}^{(r)}({\frak r}, a, \sigma_0, u)\neq \emptyset.$
Then, for all statistics $T_n(X_1,\dots, X_n),$
\begin{align}
\nonumber
\liminf_{n \rightarrow \infty}  
\inf_{T_n} \sup_{\Sigma \in \mathring{{\mathcal S}}^{(r)}({\frak r}, a, \sigma_0, u)} 
\frac{n \mathbb{E}_\Sigma (T_n(X_1,...,X_n)-\langle \theta_r(\Sigma), u \rangle )^2}{\sigma_r^2(\Sigma;u)} \geq 1.
\end{align}
Moreover, for any $\Sigma_0\in \mathring{{\mathcal S}}^{(r)}({\frak r}, a, \sigma_0, u)$
\begin{align}
\nonumber
\lim_{\varepsilon\to 0}\liminf_{n \rightarrow \infty}  
\inf_{T_n} \sup_{\Sigma \in \mathring{{\mathcal S}}^{(r)}({\frak r}, a, \sigma_0, u),
\|\Sigma-\Sigma_0\|_1 \leq \varepsilon} 
\frac{n \mathbb{E}_\Sigma (T_n(X_1,...,X_n)-\langle \theta_r(\Sigma), u \rangle )^2}{\sigma_r^2(\Sigma;u)} \geq 1.
\end{align}
\end{thm}

It follows from Corollary \ref{norm_approx_cor} and Theorem \ref{Thm Minimax lower bounds 1} that the estimator $\langle \hat \theta_r^{\delta}, u\rangle$ is efficient 
in a semi-parametric sense for quadratic loss under the assumption that 
${\frak r}_n=o(n^{1/2}).$ It turns out, however, that if 
$\frac{{\frak r}_n}{n^{1/2}}\to \infty,$ 
then not only the efficiency, but even the $\sqrt{n}$--convergence rate of this estimator fails in the class of covariance operators  ${\mathcal S}^{(r)}({\frak r}_n, a, \sigma_0, u).$

\begin{proposition}
\label{no_efficiency}
Let $a>r$ and let $\sigma_0^2$ be sufficiently small, say, 
$$\sigma_0^2\leq \frac{1}{2}\biggl[\frac{a^2}{(r-1)^2}-\frac{a}{r-1}\biggr].$$
Let ${\frak r}_n =o(n)$ and $\frac{{\frak r}_n}{n^{1/2}}\to \infty$
as $n\to \infty.$  
Then, for some constant $c=c(r;a;\sigma_0)>0$ 
$$
\lim_{n\to\infty}\sup_{\Sigma \in {\mathcal S}^{(r)}({\frak r}_n, a, \sigma_0, u)}
{\mathbb P}_{\Sigma}\bigl\{|\langle \hat \theta_r^{\delta},u\rangle -
\langle \theta_r(\Sigma),u\rangle|\geq  c \|u\| \frac{{\frak r}_n}{n}\bigr\}=1.
$$
\end{proposition}

The reason for the loss of the $\sqrt{n}$--convergence rate of plug-in estimators of linear functionals  
of principal components is their large bias in the case when the complexity of the problem is even moderately high (that is, $\frac{{\frak r}_n}{n^{1/2}}\to \infty$). 
 In \cite{KoltchinskiiLouniciPCAAHP}, a method of bias reduction in this problem 
was suggested that led to $\sqrt{n}$-consistent estimation of linear functionals. The estimator is, however, not efficient, since the basic sample split employed in its construction gives a limiting variance that is twice as large as the optimal one. Since the bias parameter depend itself on sample size in a subtle way, modifying the algorithm in \cite{KoltchinskiiLouniciPCAAHP} to obtain an efficient estimator is not straightforward, and we describe below a construction that yields an asymptotically normal estimator of $\langle \theta_r(\Sigma),u\rangle$ with optimal variance in the class 
of covariance operators ${\mathcal S}^{(r)}({\frak r}_n, a, \sigma_0, u)$ with ${\frak r}_n=o(n).$
The idea is to use only a small portion of the data (of size $o(n)$) to estimate the 
bias parameters and to use most of the data for the estimator of the target eigenvector. \\
For some $m<n/3,$ 
we split the sample $X_1,\dots, X_n$ into three disjoint subsamples, one of size
$n':=n-2m>n/3$ and two others of size $m$ each. 
In Theorem \ref{efficient_tilde} below, we choose $m=m_n=o(n)$ as $n\to\infty,$
which implies $n'=n'_n=(1+o(1))n$ as 
$n\to \infty.$ Denote by $\hat \Sigma^{(1)}, \hat \Sigma^{(2)}, \hat \Sigma^{(3)}$
the sample covariances based on these three subsamples and let 
$
\hat \theta_r^{\delta_j, j}, j=1,2,3
$
be the corresponding empirical eigenvectors 
with parameters $\delta_j=\tau \|\hat \Sigma^{(j)}\|$ for a proper choice 
of $\tau$ (see Lemma \ref{hatPdelta}). 
Let
$$
\check d_r := 
\frac{\langle \hat \theta_r^{\delta_1,1}, \hat \theta_r^{\delta_2,2}\rangle}{\langle \hat \theta_r^{\delta_2,2}, \hat \theta_r^{\delta_3,3}\rangle^{1/2}}~~~
\text{and}~~~ 
\check \theta_r := \frac{\hat \theta_r^{\delta_1,1}}{\check d_r\vee (1/2)}.
$$ 
Our main goal is to prove the following result showing the efficiency of the estimator 
$\langle \check \theta_r,u\rangle$ of the linear functional $\langle \theta_r(\Sigma),u\rangle.$ Its proof will be given in Section \ref{Sec:efficient_tilde}.

\begin{thm}
\label{efficient_tilde}
Let $u\in {\mathbb H},$ $a>1$ and $\sigma_0>0.$ Suppose that ${\frak r}_n >1$ and ${\frak r}_n =o(n)$ as $n\to\infty.$
Take $m=m_n$ such that $m_n=o(n)$ and $n {\frak r}_n =o(m_n^2)$ as $n\to\infty.$
Then 
\begin{align}
\label{norm_approx_main_tilde}
&
\sup_{\Sigma \in {\mathcal S}^{(r)}({\frak r}_n, a, \sigma_0, u)}\sup_{x\in {\mathbb R}}\bigl|{\mathbb P}_{\Sigma}\bigl\{\frac{\sqrt{n}(\langle \check \theta_r,u\rangle- \langle\theta_r(\Sigma),u\rangle)}{\sigma_r(\Sigma;u)}\leq x\bigr\}-\Phi(x)\bigr| \to 0\ {\rm as}\ n\to\infty.
\end{align}
Moreover, under Assumption \ref{assump_loss} on the loss $\ell,$
\begin{equation*}
\nonumber
\sup_{\Sigma \in {\mathcal S}^{(r)}({\frak r}_n, a, \sigma_0, u)}\bigl|{\mathbb E}_{\Sigma}\ell\biggl(\frac{\sqrt{n}(\langle \check \theta_r,u\rangle- \langle\theta_r(\Sigma),u\rangle)}{\sigma_r(\Sigma;u)}\biggr)-{\mathbb E}\ell(Z)\bigr| \to 0\ {\rm as}\ n\to\infty.
\end{equation*}
\end{thm}

\begin{remark} \normalfont 
The assumption ${\frak r}_n=o(n)$ is not necessary for the existence of a $\sqrt{n}$-consistent estimator of 
$\langle \theta_r(\Sigma),u\rangle.$ In fact, the estimator $\langle \check \theta_r,u\rangle$ (say, with $m=n/4$) 
is $\sqrt{n}$-consistent provided that ${\frak r}_n\leq cn$ for a sufficiently small constant $c>0.$ This fact easily 
follows from \eqref{tilde_theta_exponent} of Corollary \ref{tilde_theta:exponent} 
in Section \ref{Sec:efficient_tilde}. This is also the case for a somewhat simpler estimator
(based on splitting the sample into two parts) considered earlier 
by Koltchinskii and Lounici \cite{KoltchinskiiLouniciPCAAHP} (see Proposition 3). 
However, it is not clear whether asymptotically efficient estimators (in the sense of Theorem \ref{efficient_tilde}) of linear functionals $\langle \theta_r(\Sigma),u\rangle$ of the eigenvector $\theta_r(\Sigma)$ with $\sqrt{n}$-rate and  optimal limit variance $\sigma_r(\Sigma;u)$ exist when the condition 
${\frak r}_n=o(n)$ does not hold. In this case, the linear term of the perturbation series,  
that determines the limit variance $\sigma_r(\Sigma;u),$ is no longer dominant,
which makes the existence of such estimators unlikely.
However, asymptotically normal estimators 
of functionals $\langle \theta_r(\Sigma),u\rangle$ might still exist (but with a larger limit variance). It could be easier to develop such estimators
in the case of spiked covariance models rather than in the more general framework of the current paper. 
The solution of this problem would rely on the tools of random matrix theory (see, \cite{PaulStatisticaSinica}
as well as the more recent paper \cite{PCA1}) rather than perturbation theory, and, possibly, it would require the development of minimax lower bound techniques  different from those employed in the present paper.  
\end{remark}

\begin{remark} \normalfont 
It is not hard to develop similar asymptotically efficient estimators for $l$-dimensional ``functionals" of the form 
$A\theta_r(\Sigma),$ where 
$A$ is a linear operator from ${\mathbb H}$ into ${\mathbb R}^{l}$ for a fixed (small) dimension $l.$
This is equivalent to the problem of estimation of $ (\langle \theta_r(\Sigma),u_1\rangle, \dots, \langle \theta_r(\Sigma),u_l\rangle )$ for several linear functionals $u_1,\dots, u_l\in {\mathbb H}.$
The bias reduction method developed in this paper can be extended to this case 
and the proof of asymptotic normality of the resulting estimators follows along the same lines as 
in the case when $l=1$ with asymptotic covariance matrix equal to
\begin{equation*}
\left (\mu_r \langle \Sigma C_r u_i, C_r u_j \rangle  \right )_{i,j=1, \dots, p}.
\end{equation*} Similarly, our approach can be extended to linear functionals of multiple eigenvectors of multiplicity $1$ each, e.g. $(\langle \theta_r(\Sigma), u \rangle, \langle \theta_s(\Sigma), v \rangle)$, $u, v \in \mathbb{H}$. In this case the asymptotic covariance equals
\begin{equation*}
-\frac{\mu_r \mu_s}{(\mu_r-\mu_s)^2} \langle \theta_r(\Sigma), v \rangle \langle \theta_s(\Sigma), u \rangle. 
\end{equation*}
In this case the de-biasing strategy in Theorem 3.3 can be adjusted by using the second and third part oft the sample to estimate the bias for both $\theta_r(\Sigma)$ and $\theta_s(\Sigma)$. \\
However, note that when ${\bf r}(\Sigma)$ is large, the asymptotic normality of random vectors $n^{1/2}(\check \theta_r-\theta_r(\Sigma))$ holds only in the sense of finite-dimensional distributions, not in the sense of weak convergence in the Hilbert space ${\mathbb H}$ (indeed, the norm $\|\check \theta_r-\theta_r(\Sigma)\|$ is of order $\sqrt{{{\bf r}(\Sigma)/n}} \gg 1/\sqrt{n}$). 
\end{remark}

\begin{remark} \normalfont 
Our method of bias reduction does not seem 
to have an easy extension to the problem of estimation of linear functionals of spectral projections 
$P_r$ for an eigenvalue of multiplicity $>1$. In part, this was a motivation for the first author to develop
a more general approach to bias reduction (a so called ``bootstrap chain" method) and to study the problem 
of efficient estimation for more general smooth functionals of covariance of the form $\langle f(\Sigma),B\rangle,$
where $f$ is a smooth function on the real line (see \cite{Koltchinskii_2017}).
So far, the asymptotic efficiency for the resulting ``bootstrap chain" estimators has been 
proved under more restrictive assumptions on the underlying covariance $\Sigma.$ 
In particular, it was assumed that ${\mathbb H}$ is a space of finite (high) dimension $p$ 
and that the spectrum of $\Sigma$ is both upper and lower bounded away from $0$ by constants which implies that ${\bf r}(\Sigma)\asymp p.$
\end{remark}

\begin{remark} 
Lemma \ref{th:normal_tilde} of Section \ref{Sec:efficient_tilde} provides explicit bounds 
on the accuracy of the normal approximation in Theorem \ref{efficient_tilde}. Using these bounds,
it is possible to state somewhat more complicated conditions under which the normal approximation 
holds if $a=a_n\to \infty$ or $\sigma_0=\sigma_0^{(n)}\to 0$ as $n\to\infty.$
In particular, the normal approximation \eqref{norm_approx_main_tilde} still holds uniformly in ${\mathcal S}^{(r)}({\frak r}_n, a_n, \sigma_0^{(n)}, u)$ provided that $m_n=o(n)$ and 
$$
\frac{a_n^2}{\sigma_0^{(n)}}\biggl(\sqrt{\frac{n {\frak r}_n}{m_n^2} \log \frac{m_n^2}{n{\frak r}_n}}
\bigvee \sqrt{\frac{n\log^2 \frac{m_n^2}{n {\frak r}_n}}{m_n^2}}\biggr)\to 0\ {\rm as}\ n\to\infty.
$$
\end{remark}

Finally, we show that $\sigma_r(\Sigma;u)$ can be consistently estimated by 
$\sigma_r(\hat \Sigma;u),$ which allows us to replace the standard 
deviation $\sigma_r(\Sigma;u)$ in the normal approximation (\ref{norm_approx_main_tilde}) by its empirical version. This yields 
the following result that can be used for
hypotheses testing of linear functionals of $\theta_r$. 
See Section \ref{Sec:efficient_tilde_empirical} for its proof.

\begin{Corollary}
\label{efficient_tilde_empirical}
Under the conditions of Theorem \ref{efficient_tilde},
\begin{eqnarray}
\label{norm_approx_main_tilde_empirical}
&
\nonumber
\sup_{\Sigma \in {\mathcal S}^{(r)}({\frak r}_n, a, \sigma_0, u)}\sup_{x\in {\mathbb R}}\bigl|{\mathbb P}_{\Sigma}\bigl\{\frac{\sqrt{n}(\langle \check \theta_r,u\rangle- \langle\theta_r(\Sigma),u\rangle)}{\sigma_r(\hat \Sigma;u)}\leq x\bigr\}-\Phi(x)\bigr| \to 0\ {\rm as}\ n\to\infty.
\end{eqnarray}
\end{Corollary}

\section{Proof of Theorem \ref{norm_approx}}
\label{sec:norm_approx}

We will prove the result for empirical eigenvectors $\hat \theta_r$
rather than for $\hat \theta_r^{\delta}.$ The reduction to this case is based on Lemma 
\ref{hatPdelta} which immediately implies that 
$$
\sup_{\Sigma \in {\mathcal S}^{(r)}({\frak r}_n, a, \sigma_0, u)}
{\mathbb P}_{\Sigma}\{\hat \theta_r^{\delta}\neq \hat \theta_r\}
\leq e^{-\beta \tau^2 n}.
$$
Therefore, denoting 
$$
\xi_n(\Sigma):= \frac{\sqrt{n}(\langle \hat \theta_r^{\delta},u\rangle- \sqrt{1+b_r(\Sigma)}\langle \theta_r(\Sigma),u\rangle)}{\sigma_r(\Sigma;u)}
$$
and 
$$
\eta_n(\Sigma):= \frac{\sqrt{n}(\langle \hat \theta_r^,u\rangle- \sqrt{1+b_r(\Sigma)}\langle \theta_r(\Sigma),u\rangle)}{\sigma_r(\Sigma;u)},
$$
we obtain
$$
\sup_{\Sigma \in {\mathcal S}^{(r)}({\frak r}_n, a, \sigma_0, u)}\sup_{x\in {\mathbb R}}|
{\mathbb P}_{\Sigma}\{\xi_n(\Sigma)\leq x\}-
{\mathbb P}_{\Sigma}\{\eta_n(\Sigma)\leq x\}| \leq e^{-\beta \tau^2 n} \to 0\ {\rm as}\ n\to\infty.
$$
Also, since 
$
\xi_n(\Sigma)
\leq \frac{2\sqrt{n}\|u\|}{\sigma_r(\Sigma;u)}
$
and 
$
\eta_n(\Sigma)
\leq \frac{2\sqrt{n}\|u\|}{\sigma_r(\Sigma;u)},
$
we obtain that 
$$
\sup_{\Sigma \in {\mathcal S}^{(r)}({\frak r}_n, a, \sigma_0, u)}|
{\mathbb E}_{\Sigma}\ell(\xi_n(\Sigma))-{\mathbb E}_{\Sigma}\ell(\eta_n(\Sigma))| 
$$
$$
\leq \sup_{\Sigma \in {\mathcal S}^{(r)}({\frak r}_n, a, \sigma_0, u)}
{\mathbb E}_{\Sigma}|\ell(\xi_n(\Sigma))
- 
\ell(\eta_n(\Sigma))| I(\hat \theta_r^{\delta}\neq \hat \theta_r)
\leq 2\ell\biggl(\frac{2\sqrt{n}\|u\|}{\sigma_0}\biggr)e^{-\beta \tau^2 n} \to 0,
$$
under Assumption \ref{assump_loss}.

We will prove more explicit bounds for the estimator $\hat \theta_r$ stated below in Lemma \ref{th:normal} that immediately implies the result. 

Our starting point is the first order perturbation expansion of the empirical spectral projection operator $\hat P_r$:
\begin{equation}
\label{first_perturb}
\hat P_r = P_r + L_r(E) + S_r(E)
\end{equation}
with a linear term $L_r(E)=P_r E C_r + C_r E P_r$ and a remainder $S_r(E),$
where $E:=\hat \Sigma-\Sigma.$
It was proved in \cite{KoltchinskiiLouniciPCAAHP} that, under the assumption 
\begin{equation}
\label{cond_gamma}
{\mathbb E}\|\hat \Sigma-\Sigma\|\leq \frac{(1-\gamma)g_r}{2}
\end{equation}
for some $\gamma \in (0,1),$ the bilinear form of the remainder $S_r(E)$ satisfies the following 
concentration inequality: for all $u,v\in {\mathbb H}$ and for all $t\geq 1$ with probability at least $1-e^{-t}$
\begin{equation}
\label{conc_S_r}
\Bigl|\langle (S_r(E)-{\mathbb E}S_r(E))u,v\rangle\Bigr| \lesssim_{\gamma} 
\frac{\|\Sigma\|^2}{g_r^2}\biggl(\sqrt{\frac{{\bf r}(\Sigma)}{n}}\bigvee \sqrt{\frac{t}{n}}\bigvee \frac{t}{n}\biggr)\sqrt{\frac{t}{n}} \|u\|\|v\|.
\end{equation}
Under the same assumption, it was also proved in \cite{KoltchinskiiLouniciPCAAHP} that the following representation holds for the bias ${\mathbb E}\hat P_r-P_r$
of empirical spectral projections $\hat P_r:$
\begin{equation}
\label{bias_P_r}
{\mathbb E}\hat P_r-P_r = P_r({\mathbb E}\hat P_r-P_r)P_r+T_r,
\end{equation}
where the main term $P_r({\mathbb E}\hat P_r-P_r)P_r$ is aligned with the spectral projection $P_r$ and is of order
\begin{equation}
\label{bias_main}
\|P_r({\mathbb E}\hat P_r-P_r)P_r\|\lesssim \frac{\|\Sigma\|^2}{g_r^2}\frac{{\bf r}(\Sigma)}{n}
\end{equation}
and the remainder $T_r$ satisfies the bound 
\begin{equation}
\label{bd_T_r}
\|T_r\|\lesssim_{\gamma}  \frac{m_r\|\Sigma\|^2}{g_r^2} \sqrt{\frac{{\bf r}(\Sigma)}{n}}\frac{1}{\sqrt{n}}.
\end{equation}
Representation (\ref{bias_P_r}) is especially simple in the case when $P_r$ is of rank $1$ ($m_r=1$), which also implies that 
$\hat P_r$ is of rank $1.$ In this case,
$P_r= \theta_r \otimes \theta_r,$ $\hat P_r = \hat \theta_r\otimes \hat \theta_r$ for unit eigenvectors $\theta_r, \hat \theta_r$ 
of covariance operators $\Sigma, \hat \Sigma,$ respectively, and 
$$
P_r({\mathbb E}\hat P_r-P_r)P_r= b_r P_r
$$
for a ``bias parameter" $b_r = b_r(\Sigma):$
$$
b_r = {\mathbb E}\langle \hat \theta_r, \theta_r\rangle^2 -1 \in [-1,0].
$$
Thus, it follows from (\ref{bias_P_r}) that 
\begin{equation}
\label{EhatP_r}
{\mathbb E} \hat P_r = (1+b_r)P_r + T_r.
\end{equation}

We obtain from (\ref{first_perturb}) and (\ref{EhatP_r}) that 
\begin{equation}
\label{center_b}
\hat P_r - (1+b_r)P_r = L_r(E)+ S_r(E)-{\mathbb E}S_r(E) +T_r.
\end{equation}
Denote 
$$
\rho_r(u):= \langle (\hat P_r - (1+b_r)P_r)\theta_r, u\rangle, u\in {\mathbb H}. 
$$
As in \cite{KoltchinskiiLouniciPCAAHP},
the function $\rho_r(u), u\in {\mathbb H}$ will be used in what follows to control the linear forms 
$\langle \hat \theta_r- \sqrt{1+b_r}\theta_r,u\rangle, u\in {\mathbb H}.$ First, we need to derive 
some bounds on $\rho_r(u).$

The following lemma is an immediate consequence of (\ref{center_b}), (\ref{conc_S_r}) and (\ref{bd_T_r}).

\begin{lemma}
Suppose condition (\ref{cond_gamma}) holds for some $\gamma\in (0,1).$
Then, for all $u\in {\mathbb H}$ and for all $t\geq 1$ with probability at least $1-e^{-t}$
\begin{equation}
\label{rho-L}
|\rho_r(u)-\langle L_r(E)\theta_r,u\rangle| \lesssim_{\gamma} 
\frac{\|\Sigma\|^2}{g_r^2}\biggl(\sqrt{\frac{{\bf r}(\Sigma)}{n}}\bigvee \sqrt{\frac{t}{n}}\bigvee \frac{t}{n}\biggr)\sqrt{\frac{t}{n}} \|u\|. 
\end{equation}
\end{lemma}

We will need simple concentration and normal approximation bounds for $\langle L_r(E)\theta_r,u\rangle$ given 
in the next lemma.

\begin{lemma}
\label{Lemma:conc_L_r}
For all $t\geq 1$ with probability at least $1-e^{-t}$
\begin{equation}
\label{conc_L_r}
|\langle L_r(E)\theta_r,u\rangle| \lesssim \sigma_r (\Sigma;u) \biggl(\sqrt{\frac{t}{n}}\bigvee \frac{t}{n}\biggr).
\end{equation}
Moreover, if $\sigma_r(\Sigma;u)>0,$ then
\begin{equation}
\label{normal_approx_L_r}
\sup_{x\in {\mathbb R}}\bigl|{\mathbb P}\bigl\{\frac{\sqrt{n}\langle L_r(E)\theta_r,u\rangle}{\sigma_r(\Sigma;u)}\leq x\bigr\}-\Phi(x)\bigr|
\lesssim \frac{1}{\sqrt{n}},
\end{equation}
where $\Phi$ is the distribution function of standard normal r.v.
\end{lemma}

\begin{proof}
Without loss of generality, assume that the space ${\mathbb H}$ is finite-dimensional (the general case follows by a simple approximation argument). 
Since $L_r(E)= P_rEC_r+C_rEP_r$ and $C_r\theta_r=0,$ we have 
$$
\langle L_r(E)\theta_r, u\rangle= \langle C_rEP_r \theta_r, u\rangle = \langle E\theta_r, C_r u\rangle = \langle E, \theta_r \otimes C_r u\rangle. 
$$

Since $E$ is self-adjoint, we obtain that 
$$
\langle L_r(E)\theta_r, u\rangle = \frac{1}{2}\langle E, \theta_r \otimes C_r u+ C_r u \otimes \theta_r\rangle.
$$
Let $Z, Z_1,\dots, Z_n$ be i.i.d. standard normal vectors in ${\mathbb H}$ such that $X_j = \Sigma^{1/2} Z_j.$
Then 
$$
E = \Sigma^{1/2} \biggl(n^{-1}\sum_{j=1}^n Z_j\otimes Z_j -{\mathbb E}(Z\otimes Z)\biggr)\Sigma^{1/2}. 
$$
Defining 
$$
D:= \frac{1}{2}\Sigma^{1/2}(\theta_r \otimes C_r u+ C_r u \otimes \theta_r)\Sigma^{1/2}
= \frac{1}{2}\bigl(\Sigma^{1/2}\theta_r \otimes \Sigma^{1/2}C_ru + \Sigma^{1/2}C_r u\otimes \Sigma^{1/2}\theta_r\bigr),
$$ 
we obtain that
\begin{align*}
\langle L_r(E)\theta_r, u\rangle &  = \biggl\langle n^{-1}\sum_{j=1}^n Z_j\otimes Z_j -{\mathbb E}(Z\otimes Z), D\biggr\rangle \\
& = n^{-1}\sum_{j=1}^n (\langle DZ_j,Z_j\rangle - {\mathbb E}\langle DZ,Z\rangle).
\end{align*}
Clearly,
$
\langle DZ,Z\rangle \stackrel{d}{=} \sum_{k} \lambda_k g_k^2, 
$
where $\{\lambda_k\}$ are the eigenvalues of $D$ and $\{g_k\}$ are i.i.d.
standard normal r.v. It easily follows that 
$$
{\mathbb E} \langle DZ,Z\rangle = {\rm tr}(D)=0
$$
and 
$$
{\rm Var}(\langle DZ,Z\rangle)= 2\sum_{k}\lambda_k^2 = 2\|D\|_2^2
=\sigma_r^2 (\Sigma;u).
$$ 
We can now represent $\langle L_r(E)\theta_r, u\rangle$ as follows:
$$
\langle L_r(E)\theta_r, u\rangle \stackrel{d}{=} n^{-1}\sum_{j=1}^n \sum_{k} \lambda_k (g_{k,j}^2-1),
$$
where $\{g_{k,j}\}$ are i.i.d. standard normal r.v. Using standard exponential bounds for sums of independent $\psi_1$ r.v. 
(see, e.g., \cite{Vershynin}, Proposition 5.16 or Theorem 3.1.9 in \cite{GN16}), we obtain that with probability at least $1-e^{-t}$
$$
\bigl|n^{-1}\sum_{j=1}^n \sum_{k} \lambda_k (g_{k,j}^2-1)\bigr| \lesssim 
\biggl(\sum_{k}\lambda_k^2\biggr)^{1/2}\sqrt{\frac{t}{n}}  \bigvee \sup_{k}|\lambda_k|\frac{t}{n},
$$
which implies that with the same probability 
$$
|\langle L_r(E)\theta_r, u\rangle| 
\lesssim 
\|D\|_2\sqrt{\frac{t}{n}}  \bigvee \|D\|\frac{t}{n}.
$$
Since $\|D\|\leq \|D\|_2 = \frac{1}{2} \sigma_r^{2}(\Sigma;u),$ bound (\ref{conc_L_r}) follows. 

To prove (\ref{normal_approx_L_r}), we use the Berry-Esseen bound that implies 
$$
\sup_{x\in {\mathbb R}}\bigl|{\mathbb P}\bigl\{    
\frac{\sum_{j=1}^n \sum_{k} \lambda_k (g_{k,j}^2-1)}
{\sqrt{n} (2\sum_{k}\lambda_k^2)^{1/2}}
\leq x\bigr\}-\Phi(x)\bigr|
\lesssim \frac{\sum_{k}|\lambda_k|^3}{(\sum_{k}\lambda_k^2)^{3/2}}\frac{1}{\sqrt{n}},
$$ 
and therefore
$$
\sup_{x\in {\mathbb R}}\bigl|{\mathbb P}\bigl\{\frac{\sqrt{n}\langle L_r(E)\theta_r,u\rangle}{\sigma_r(\Sigma;u)}\leq x\bigr\}-\Phi(x)\bigr|
\lesssim \frac{\|D\|_3^3}{\|D\|_2^3} \frac{1}{\sqrt{n}}\lesssim \frac{\|D\|}{\|D\|_2}\frac{1}{\sqrt{n}}\lesssim 
\frac{1}{\sqrt{n}}.
$$
\end{proof}

The following bounds on $\rho_r(u)$ immediately follow from (\ref{rho-L}) and (\ref{conc_L_r}).

\begin{lemma}
Suppose condition (\ref{cond_gamma}) holds for some $\gamma\in (0,1).$
Then, for all $u\in {\mathbb H}$ and for all $t\geq 1$ with probability at least $1-e^{-t}$
\begin{eqnarray}
\label{bd_rho}
&
|\rho_r(u)| \lesssim_{\gamma} 
\sigma_r(\Sigma;u)
\biggl(\sqrt{\frac{t}{n}}\bigvee \frac{t}{n}\biggr)+
\frac{\|\Sigma\|^2}{g_r^2}\biggl(\sqrt{\frac{{\bf r}(\Sigma)}{n}}\bigvee \sqrt{\frac{t}{n}}\biggr)\sqrt{\frac{t}{n}} \|u\|.
\end{eqnarray}
Moreover, with the same probability 
\begin{eqnarray}
\label{bd_rho_other}
&
|\rho_r(u)| 
\lesssim_{\gamma}
\frac{\|\Sigma\|}{g_r}\sqrt{\frac{t}{n}}\|u\|+
\frac{\|\Sigma\|^2}{g_r^2}\biggl(\sqrt{\frac{{\bf r}(\Sigma)}{n}}\bigvee \sqrt{\frac{t}{n}}\biggr)\sqrt{\frac{t}{n}} \|u\|.
\end{eqnarray}
and, for $u=\theta_r,$ 
\begin{equation}
\label{bd_rho_theta}
|\rho_r(\theta_r)| \lesssim_{\gamma} 
\frac{\|\Sigma\|^2}{g_r^2}\biggl(\sqrt{\frac{{\bf r}(\Sigma)}{n}}\bigvee \sqrt{\frac{t}{n}}\biggr)\sqrt{\frac{t}{n}}.
\end{equation}
\end{lemma}

Note that we dropped the term $\frac{t}{n}$ in some of the expressions on the right hand side of the above bounds (compare with (\ref{rho-L})). This term is dominated by $\sqrt{\frac{t}{n}}$
for $t\leq n.$ Moreover, it follows from the definition of $\rho_r(u)$ that it is upper bounded by $2\|u\|.$ 
Since $\frac{\|\Sigma\|}{g_r}\geq 1,$ this easily implies that, for $t\geq n,$ the right hand side of bound (\ref{bd_rho_other})
(with a proper constant) is larger than $|\rho_r(u)|.$
Bound (\ref{bd_rho_theta}) follows from (\ref{rho-L}) since $\langle L_r(E) \theta_r, \theta_r\rangle=0.$

To study concentration and normal approximation of the linear form
$$\langle \hat \theta_r- \sqrt{1+b_r}\theta_r,u\rangle, u\in {\mathbb H},$$
it remains to prove that it can be approximated by $\langle L_r(E)\theta_r,u\rangle.$

\begin{lemma}
Suppose that for some $\gamma\in (0,1)$ condition (\ref{cond_gamma}) holds and, in addition, 
 \begin{equation}
 \label{cond_b_r}
 1+b_r \geq \gamma.
 \end{equation} 
 Then, for all $u\in {\mathbb H}$ and for all $t\geq 1,$ with probability at least $1-e^{-t}$
 \begin{equation}
\label{theta-L}
|\langle \hat \theta_r- \sqrt{1+b_r}\theta_r,u\rangle- \langle L_r(E)\theta_r,u\rangle|
\lesssim_{\gamma} \frac{\|\Sigma\|^2}{g_r^2}\biggl(\sqrt{\frac{{\bf r}(\Sigma)}{n}}\bigvee \sqrt{\frac{t}{n}}\bigvee \frac{t}{n}\biggr)\sqrt{\frac{t}{n}}\|u\|.
\end{equation}
\end{lemma}

\begin{proof}
We use the following representation obtained in \cite{KoltchinskiiLouniciPCAAHP} (see (6.7) in \cite{KoltchinskiiLouniciPCAAHP}), which holds provided that $\hat \theta_r$ and $\theta_r$
are properly aligned so that $\langle \hat \theta_r,\theta_r\rangle \geq 0:$
\begin{eqnarray}
\label{represent_theta_r}
&
\langle \hat \theta_r- \sqrt{1+b_r}\theta_r,u\rangle 
= \frac{\rho_r(u)}{\sqrt{1+b_r+\rho_r(\theta_r)}}
\\
&
\nonumber
 - \frac{\sqrt{1+b_r}}{\sqrt{1+b_r+\rho_r(\theta_r)}(\sqrt{1+b_r+\rho_r(\theta_r)}+\sqrt{1+b_r})}\rho_{r}(\theta_r)\langle \theta_r,u\rangle
\end{eqnarray}
(it is clear from the proof given in \cite{KoltchinskiiLouniciPCAAHP} that $1+b_r+\rho_r(\theta_r)\geq 0$). Denote 
$$
\nu_r := \frac{\rho_r(\theta_r)}{1+b_r}.
$$
Then, it is easy to see that 
\begin{align}
\label{theta_rho}
\langle \hat \theta_r- \sqrt{1+b_r}\theta_r,u\rangle= \rho_r(u)& - \frac{b_r/(1+b_r)+\nu_r}{1+\nu_r+ \sqrt{(1+\nu_r)/(1+b_r)}}\rho_r(u)
\\
&
\notag 
-\frac{\nu_r\sqrt{1+b_r}}{1+\nu_r+ \sqrt{1+\nu_r}}\langle \theta_r,u\rangle.
\end{align}
Recall that (\ref{cond_gamma}) and (\ref{cond_b_r})
hold for some $\gamma \in (0,1).$ 
If $|\nu_r|\leq 1/2,$
then (\ref{theta_rho}) easily implies that 
\begin{equation}
\label{theta_rho_1}
|\langle \hat \theta_r- \sqrt{1+b_r}\theta_r,u\rangle- \rho_r(u)| \leq \frac{1}{\gamma} (|b_r|+|\nu_r|) 
|\rho_r(u)| + |\nu_r| |\langle \theta_r,u\rangle|.
\end{equation}
It also follows from (\ref{bd_rho_theta}) that, under condition (\ref{cond_b_r}),
\begin{equation}
\label{bd_nu_r}
|\nu_r| \lesssim_{\gamma} 
\frac{\|\Sigma\|^2}{g_r^2}\biggl(\sqrt{\frac{{\bf r}(\Sigma)}{n}}\bigvee \sqrt{\frac{t}{n}}\biggr)\sqrt{\frac{t}{n}} 
\end{equation}
with probability at least $1-e^{-t}.$ On the other hand, bound (\ref{bias_main}) implies that 
\begin{equation}
\label{bd_b_r}
|b_r|\lesssim \frac{\|\Sigma\|^2}{g_r^2}\frac{{\bf r}(\Sigma)}{n}.
\end{equation}
It follows from (\ref{bd_nu_r}) that for the condition $|\nu_r|\leq 1/2$ to hold with probability at least $1-e^{-t},$ it is enough 
to have
\begin{equation}
\label{bd_nu_r_A}
\frac{\|\Sigma\|^2}{g_r^2}\biggl(\sqrt{\frac{{\bf r}(\Sigma)}{n}}\bigvee \sqrt{\frac{t}{n}}\biggl)\sqrt{\frac{t}{n}}\leq c_{\gamma}
\end{equation}
for a small enough constant $c_{\gamma}>0.$ Assume that (\ref{bd_nu_r_A}) holds. Note also that it implies that $t\lesssim n$
and condition (\ref{cond_gamma}) and Theorem \ref{KL-Bernoulli} imply that $\frac{\|\Sigma\|}{g_r}\sqrt{\frac{{\bf r}(\Sigma)}{n}}\lesssim 1.$
It follows from (\ref{theta_rho_1}), (\ref{bd_rho_other}), (\ref{bd_nu_r}) and (\ref{bd_b_r}) that with probability at least $1-3e^{-t}$:
\begin{eqnarray}
\label{bd_Long}
&
\nonumber
|\langle \hat \theta_r- \sqrt{1+b_r}\theta_r,u\rangle- \rho_r(u)| 
\\
&
\nonumber
\lesssim_{\gamma} 
\biggr[
\frac{\|\Sigma\|^2}{g_r^2}\frac{{\bf r}(\Sigma)}{n}
+ 
\biggl(
\frac{\|\Sigma\|^2}{g_r^2}\biggl(\sqrt{\frac{{\bf r}(\Sigma)}{n}}\bigvee \sqrt{\frac{t}{n}}\biggr)\sqrt{\frac{t}{n}}
\biggr)
\wedge 1/2
\biggl]
\\
&
\nonumber
\times 
\biggl[ 
\frac{\|\Sigma\|}{g_r}\sqrt{\frac{t}{n}}\|u\|+ \frac{\|\Sigma\|^2}{g_r^2}\biggl(\sqrt{\frac{{\bf r}(\Sigma)}{n}}\bigvee \sqrt{\frac{t}{n}}\biggr)\sqrt{\frac{t}{n}}\|u\|
\biggr]
\\
&
+\frac{\|\Sigma\|^2}{g_r^2}\biggl(\sqrt{\frac{{\bf r}(\Sigma)}{n}}\bigvee \sqrt{\frac{t}{n}}\biggr)\sqrt{\frac{t}{n}}\|u\|.
\end{eqnarray}
Using the facts that 
$$
\frac{\|\Sigma\|^2}{g_r^2}\frac{{\bf r}(\Sigma)}{n}\lesssim 
\frac{\|\Sigma\|}{g_r}\sqrt{\frac{{\bf r}(\Sigma)}{n}}\lesssim 1,
$$
that 
$$
\frac{\|\Sigma\|^2}{g_r^2}\frac{t}{n}\lesssim \frac{\|\Sigma\|}{g_r}\sqrt{\frac{t}{n}}\lesssim 1
$$
and that 
$$
\frac{\|\Sigma\|^2}{g_r^2}\sqrt{\frac{{\bf r}(\Sigma)}{n}}\sqrt{\frac{t}{n}}
\lesssim 
\frac{\|\Sigma\|}{g_r}\biggl(\frac{{\bf r}(\Sigma)}{n}\biggr)^{1/4}\biggl(\frac{t}{n}\biggr)^{1/4}
\leq \frac{\|\Sigma\|}{g_r} \biggl(\sqrt{\frac{{\bf r}(\Sigma)}{n}}\bigvee \sqrt{\frac{t}{n}}\biggr)
$$
(that follow from condition (\ref{bd_nu_r_A})), it is easy to conclude that the last term
$$
\frac{\|\Sigma\|^2}{g_r^2}\biggl(\sqrt{\frac{{\bf r}(\Sigma)}{n}}\bigvee \sqrt{\frac{t}{n}}\biggr)\sqrt{\frac{t}{n}}\|u\|
$$
in the right hand side of bound (\ref{bd_Long}) is dominant. Hence, with probability at least $1-e^{-t}$
\begin{equation}
|\langle \hat \theta_r- \sqrt{1+b_r}\theta_r,u\rangle- \rho_r(u)| \lesssim_{\gamma}
\frac{\|\Sigma\|^2}{g_r^2}\biggl(\sqrt{\frac{{\bf r}(\Sigma)}{n}}\bigvee \sqrt{\frac{t}{n}}\biggr)\sqrt{\frac{t}{n}}\|u\|
\end{equation}
provided that condition (\ref{bd_nu_r_A}) holds. On the other hand, if 
\begin{equation*}
\frac{\|\Sigma\|^2}{g_r^2}\biggl(\sqrt{\frac{{\bf r}(\Sigma)}{n}}\bigvee \sqrt{\frac{t}{n}}\biggr)\sqrt{\frac{t}{n}}> c_{\gamma},
\end{equation*}
then
$$
|\langle \hat \theta_r- \sqrt{1+b_r}\theta_r,u\rangle- \rho_r(u)| \leq |\langle \hat \theta_r- \sqrt{1+b_r}\theta_r,u\rangle|+|\rho_r(u)| 
$$
$$
\leq (\|\hat \theta_r\|+ \sqrt{1+b_r}\|\theta_r\|)\|u\| + (\|\hat P_r\|+ (1+b_r)\|P_r\|) \|\theta_r\|\|u\|
\leq 4\|u\|
$$
$$
\lesssim_{\gamma} \frac{\|\Sigma\|^2}{g_r^2}\biggl(\sqrt{\frac{{\bf r}(\Sigma)}{n}}\bigvee \sqrt{\frac{t}{n}}\biggr)\sqrt{\frac{t}{n}}\|u\|.
$$
Thus, we proved that with probability at least $1-e^{-t}$
\begin{equation}
\label{theta-rho}
|\langle \hat \theta_r- \sqrt{1+b_r}\theta_r,u\rangle- \rho_r(u)|
\lesssim_{\gamma} \frac{\|\Sigma\|^2}{g_r^2}\biggl(\sqrt{\frac{{\bf r}(\Sigma)}{n}}\bigvee \sqrt{\frac{t}{n}}\biggr)\sqrt{\frac{t}{n}}\|u\|.
\end{equation}
It remains to combine this with the bound (\ref{rho-L}) to complete the proof.

\end{proof}

The following result is a slightly improved version 
of Theorem 6 in \cite{KoltchinskiiLouniciPCAAHP}.

\begin{lemma}
Under conditions (\ref{cond_gamma}) and (\ref{cond_b_r}) for some $\gamma \in (0,1),$
the following bounds hold for all $t\geq 1$ with probability at least $1-e^{-t}:$
\begin{equation}
\label{on_hat_theta}
|\langle \hat \theta_r- \sqrt{1+b_r}\theta_r,u\rangle| \lesssim_{\gamma}\frac{\|\Sigma\|}{g_r}\sqrt{\frac{t}{n}}\|u\|
\end{equation}
and 
\begin{equation}
\label{theta_theta_r}
|\langle \hat \theta_r- \sqrt{1+b_r}\theta_r,\theta_r\rangle|
\lesssim_{\gamma} \frac{\|\Sigma\|^2}{g_r^2}\biggl(\sqrt{\frac{{\bf r}(\Sigma)}{n}}\bigvee \sqrt{\frac{t}{n}}\biggr)\sqrt{\frac{t}{n}}.
\end{equation}
\end{lemma}

\begin{proof}
Indeed, it follows from (\ref{theta-L}) and (\ref{conc_L_r}) that, for some constants $C,C_{\gamma}>0$ with probability at least $1-e^{-t}$
$$
|\langle \hat \theta_r- \sqrt{1+b_r}\theta_r,u\rangle| \leq C \sigma_r (\Sigma;u) \biggl(\sqrt{\frac{t}{n}}\bigvee \frac{t}{n}\biggr)
+ C_{\gamma}\frac{\|\Sigma\|^2}{g_r^2}\biggl(\sqrt{\frac{{\bf r}(\Sigma)}{n}}\bigvee \sqrt{\frac{t}{n}}\biggr)\sqrt{\frac{t}{n}}\|u\|.
$$
Since $ \sigma_r (\Sigma;u) \lesssim \frac{\|\Sigma\|}{g_r}\|u\|,$ with the same probability 
$$
|\langle \hat \theta_r- \sqrt{1+b_r}\theta_r,u\rangle| \leq C  \frac{\|\Sigma\|}{g_r}\sqrt{\frac{t}{n}}\|u\|
+ C_{\gamma}\frac{\|\Sigma\|^2}{g_r^2}\biggl(\sqrt{\frac{{\bf r}(\Sigma)}{n}}\bigvee \sqrt{\frac{t}{n}}\biggr)\sqrt{\frac{t}{n}}\|u\|.
$$
We dropped the term $\frac{t}{n}$ present in bounds (\ref{theta-L}) and (\ref{conc_L_r}) since for $t\geq n$ (the only case when it is needed), the right hand side already dominates the left hand side
(which is smaller than $2\|u\|$).  
Note that condition (\ref{cond_gamma}) and Theorem \ref{KL-Bernoulli} imply that 
$
\frac{\|\Sigma\|}{g_r} \sqrt{\frac{{\bf r}(\Sigma)}{n}}\leq c_{\gamma}
$
for some constant $c_{\gamma}>0.$
Assuming that also 
$
\frac{\|\Sigma\|}{g_r}\sqrt{\frac{t}{n}}\leq c_{\gamma},
$
which implies that $t\lesssim n,$ we obtain that for some constant $C_{\gamma}>0$ with probability at least $1-e^{-t}$ bound 
(\ref{on_hat_theta}) holds.
On the other hand, if 
$
\frac{\|\Sigma\|}{g_r}\sqrt{\frac{t}{n}}> c_{\gamma},
$
then 
$$
|\langle \hat \theta_r- \sqrt{1+b_r}\theta_r,u\rangle|\leq (\|\hat \theta_r\|+ \sqrt{1+b_r}\|\theta_r\|)\|u\| \leq 2 \|u\|
\lesssim_{\gamma} \frac{\|\Sigma\|}{g_r}\sqrt{\frac{t}{n}}\|u\|,
$$
implying again (\ref{on_hat_theta}). 
For $u=\theta_r,$ $\langle L_r(E)\theta_r,u\rangle=0$ and bound (\ref{theta-L}) implies that with probability at least $1-e^{-t}$
(\ref{theta_theta_r}) holds.
\end{proof}

The following two lemmas will be used to derive normal approximation bounds for $\langle \hat \theta_r- \sqrt{1+b_r}\theta_r,u\rangle$
from the corresponding bounds for $\langle L_r(E)\theta_r,u\rangle$
as well as to control the risk for loss functions satisfying Assumption \ref{assump_loss}.
We state them without proofs (which are elementary).

\begin{lemma} 
\label{lemma_xi_eta}
For random variables $\xi,\eta,$ denote 
$$
\Delta (\xi;\eta):= \sup_{x\in {\mathbb R}}|{\mathbb P}\{\xi \leq x\}-{\mathbb P}\{\eta\leq x\}|
$$ 
and 
$$
\delta (\xi;\eta):= \inf\{\delta>0: {\mathbb P}\{|\xi-\eta|\geq \delta\}+\delta\}.
$$
Then, for a standard normal r.v. $Z,$
$$
\Delta (\xi;Z)\leq \Delta (\eta;Z)+\delta(\xi;\eta).
$$
Under Assumption \ref{assump_loss}, for all $A>0$
$$
|{\mathbb E}\ell(\xi)-{\mathbb E}\ell(\eta)|\leq 4\ell(A)\Delta(\xi;\eta)+ {\mathbb E}\ell(\xi)I(|\xi|\geq A)+ {\mathbb E}\ell(\eta)I(|\eta|\geq A).
$$
\end{lemma}




\begin{lemma}
\label{Lemma:Bernstein-loss}
Let $\xi$ be a random variable such that for some $\tau_1\geq 0$ and $\tau_2\geq 0$
and for all $t\geq 1$ with probability at least $1-e^{-t}$
$$
|\xi| \leq \tau_1 \sqrt{t} \vee \tau_2 t.
$$
Let $\ell$ be a loss function satisfying Assumption \ref{assump_loss}.
If $2c_2 \tau_2<1,$ then 
\begin{equation}
\label{Bernstein-loss}
{\mathbb E} \ell^2 (\xi)\leq 
2e\sqrt{2\pi}c_1^2 e^{2c_2^2 \tau_1^2}+ \frac{e c_1^2}{1-2c_2 \tau_2}.
\end{equation}
\end{lemma}


Next we prove the normal approximation bounds for linear forms $\langle \hat \theta_r- \sqrt{1+b_r}\theta_r,u\rangle.$

\begin{lemma}
\label{th:normal}
Suppose that conditions (\ref{cond_gamma}) and (\ref{cond_b_r}) hold for some $\gamma\in (0,1)$
and also that $n\geq 2{\bf r}(\Sigma).$
Assume that, for some $u\in {\mathbb H},$ $\sigma_r(\Sigma;u)>0.$ 
Let $\alpha\geq 1.$ Then the following bound holds: 
for some constants $C,C_{\gamma, \alpha}>0,$
\begin{align}
\label{norm_approx_A}
&
\nonumber
\sup_{x\in {\mathbb R}}\bigl|{\mathbb P}\bigl\{\frac{\sqrt{n}\langle \hat \theta_r- \sqrt{1+b_r}\theta_r,u\rangle}{\sigma_r(\Sigma;u)}\leq x\bigr\}-\Phi(x)\bigr|
\\
\leq  & C n^{-1/2} + \frac{C_{\gamma,\alpha}}{\sigma_r(\Sigma;u)}
\frac{\|\Sigma\|^2}{g_r^2}\biggl(\sqrt{\frac{{\bf r}(\Sigma)}{n} \log \frac{n}{{\bf r}(\Sigma)}}\bigvee \frac{\log \frac{n}{{\bf r}(\Sigma)}}{\sqrt{n}}\biggr)\|u\|+ \biggl(\frac{{\bf r}(\Sigma)}{n}\biggr)^{\alpha}.
\end{align}
Moreover, under Assumption \ref{assump_loss} on the loss $\ell,$ there exist constants $C, C_{\gamma}, C_{\gamma,\alpha}>0$ such that  
\begin{align}
\label{bdell}
&
\nonumber
\bigl|{\mathbb E} \ell\biggl( \frac{\sqrt{n}\langle \hat \theta_r- \sqrt{1+b_r}\theta_r,u\rangle}{\sigma_r(\Sigma;u)}\biggr)-{\mathbb E}\ell(Z)\bigr|
\\
&
\nonumber
\leq c_1 e^{c_2 A}
 \biggl(C n^{-1/2} + \frac{C_{\gamma,\alpha}}{\sigma_r(\Sigma;u)}
\frac{\|\Sigma\|^2}{g_r^2}\biggl(\sqrt{\frac{{\bf r}(\Sigma)}{n} \log \frac{n}{{\bf r}(\Sigma)}}\bigvee \frac{\log \frac{n}{{\bf r}(\Sigma)}}{\sqrt{n}}\biggr)\|u\|+ \biggl(\frac{{\bf r}(\Sigma)}{n}\biggr)^{\alpha}\biggr)
\\
&
+
2 e^{3/2} (2\pi)^{1/4} c_1 
e^{c_2^2 \tau^2}e^{-A^2/2\tau^2}
+ c_1 e^{c_2^2} e^{-A^2/4},
\end{align}
where 
$$
\tau:= C_{\gamma} \frac{\|\Sigma\| \|u\|}{g_r \sigma_r(\Sigma;u)}.
$$ 
\end{lemma}

\begin{proof}
We will use the first claim of Lemma \ref{lemma_xi_eta} with 
$$
\xi := \frac{\sqrt{n}\langle \hat \theta_r- \sqrt{1+b_r}\theta_r,u\rangle}{\sigma_r(\Sigma;u)}\ {\rm and}\ \eta:= \frac{\sqrt{n}\langle L_r(E)\theta_r,u\rangle}{\sigma_r(\Sigma;u)}.
$$
It follows from bound (\ref{theta-L}) that, under conditions (\ref{cond_gamma}) and (\ref{cond_b_r}),  
for some $C_{\gamma}>0$ 
$$
\delta(\xi;\eta)
\leq \inf_{t\geq 1}\biggl\{\frac{C_{\gamma}}{\sigma_r(\Sigma;u)}
\frac{\|\Sigma\|^2}{g_r^2}\biggl(\sqrt{\frac{{\bf r}(\Sigma)}{n}}\bigvee \sqrt{\frac{t}{n}}\bigvee \frac{t}{n}\biggr)\sqrt{t}\|u\|+ e^{-t}\biggr\}.
$$
Taking $t:= \alpha\log \big ( \frac{n}{{\bf r}(\Sigma)} \big )$ with some $\alpha\geq 1$ easily yields an upper bound 
$$
\delta (\xi;\eta)\leq \frac{C_{\gamma,\alpha}}{\sigma_r(\Sigma;u)}
\frac{\|\Sigma\|^2}{g_r^2}\biggl(\sqrt{\frac{{\bf r}(\Sigma)}{n} \log \frac{n}{{\bf r}(\Sigma)}}\bigvee \frac{\log \frac{n}{{\bf r}(\Sigma)}}{\sqrt{n}}\biggr)\|u\|+ \biggl(\frac{{\bf r}(\Sigma)}{n}\biggr)^{\alpha}.
$$
Using bound (\ref{normal_approx_L_r}) to control $\Delta(\eta;Z),$ we obtain from Lemma \ref{lemma_xi_eta} that bound \eqref{norm_approx_A} holds with some constants $C, C_{\gamma,\alpha}>0.$
To prove the second statement, we use the second bound of Lemma \ref{lemma_xi_eta} with the random variable $\xi := \frac{\sqrt{n}\langle \hat \theta_r- \sqrt{1+b_r}\theta_r,u\rangle}{\sigma_r(\Sigma;u)}$ and $\eta=Z.$
The following exponential bound on $\xi$ is an easy corollary of bound
(\ref{on_hat_theta}): 
for some constant $C_{\gamma}>0$ and for all $t\geq 1$ with probability at least 
$1-e^{-t}$
\begin{equation}
\label{on_xi}
|\xi|\leq C_{\gamma}\frac{\|\Sigma\|}{g_r \sigma_r(\Sigma;u)}\sqrt{t}\|u\|
=\tau \sqrt{t}.
\end{equation}
Using bound \eqref{Bernstein-loss} with $\tau_1=\tau$ and $\tau_2=0,$ we 
obtain 
$$
{\mathbb E} \ell^2 (\xi)\leq 
2e\sqrt{2\pi}c_1^2 e^{2c_2^2 \tau_1^2}+ e c_1^2\leq 4e\sqrt{2\pi}c_1^2 e^{2c_2^2 \tau_1^2}
$$
Therefore, 
$$
{\mathbb E} \ell(\xi)I(|\xi|\geq A)\leq  {\mathbb E}^{1/2}\ell^2 (\xi) {\mathbb P}^{1/2}\{|\xi|\geq A\}\leq 
2 e^{3/2} (2\pi)^{1/4} c_1 
e^{c_2^2 \tau^2}e^{-A^2/2\tau^2}.
$$
We also have 
$$
{\mathbb E} \ell(Z)I(|Z|\geq A)\leq c_1 e^{c_2^2} e^{-A^2/4}.  
$$
Using bound (\ref{norm_approx_A}), we can now deduce bound \eqref{bdell} from 
the second statement of Lemma \ref{lemma_xi_eta}.
\end{proof}

Lemma \ref{th:normal} immediately implies Theorem \ref{norm_approx}
(by passing to the limit as $n\to\infty$ in \eqref{norm_approx_A} and as $n\to\infty$ and then 
$A\to\infty$ in \eqref{bdell}).

\subsection{{ \bf Proof of Proposition \ref{no_efficiency}}}

Denote 
$$
A_r(\Sigma):= 2\trace(P_r\Sigma P_r)\trace(C_r\Sigma C_r)= 
2\sum_{s\neq r}\frac{\mu_r \mu_s m_s}{(\mu_r-\mu_s)^2}.
$$
It was shown in \cite{KoltchinskiiLouniciPCAAOS} that 
$$
{\mathbb E}\|L_r(E)\|_2^2= \frac{A_r(\Sigma)}{n},
$$
where $E=\hat \Sigma-\Sigma.$
Note that 
\begin{equation}
\label{bd_Ar_up}
\frac{A_r(\Sigma)}{2} \leq \frac{\mu_r}{g_r^2} (\trace(\Sigma)-\mu_r)
\leq \frac{\|\Sigma\|^2}{g_r^2} {\bf r}(\Sigma)
\end{equation}
and 
\begin{equation}
\label{bd_Ar_down}
\frac{A_r(\Sigma)}{2} \geq \frac{\mu_1 \mu_r}{(\mu_1-\mu_r)^2 \vee \mu_r^2} 
({\bf r}(\Sigma)-1).
\end{equation}

\begin{lemma}
\label{brbou}
The following representation holds:
$$
b_r(\Sigma)= -\frac{1}{2}\frac{A_r(\Sigma)}{n} + \beta_r,
$$ 
where 
$$
|\beta_r| \lesssim \frac{\|\Sigma\|^3 }{g_r^3} 
\biggl(\sqrt{\frac{{\bf r}(\Sigma)}{n}} \bigvee 
\frac{{\bf r}(\Sigma)}{n}\biggr)^3.
$$
\end{lemma}

\begin{proof}
Recall representation \eqref{bias_P_r} and bound \eqref{bd_T_r}. 
Note that 
$$
b_r = \trace (P_r({\mathbb E}\hat P_r-P_r)P_r)
$$
and 
$$
{\mathbb E}\hat P_r -P_r = {\mathbb E}S_r(E).
$$
We will use the following representation for $S_r(E)$ (based on perturbation series 
for $\hat P_r$) that 
easily follows from Lemma 4 in \cite{KoltchinskiiLouniciPCAarxiv}:
$$
S_r(E) = P_r E C_r E C_r + C_r E P_r E C_r + C_r E C_r E P_r 
$$
$$
- P_r E P_r E C_r^2 - P_r E C_r^2 E P_r - C_r^2 E P_r E P_r
+  S_r^{(3)}(E),
$$
where 
$$
\|S_r^{(3)}(E)\| \lesssim \frac{\|E\|^3}{g_r^3}.
$$
Since $P_r C_r = C_r P_r =0$ this implies 
$$
P_r S_r(E) P_r = -P_r E C_r^2 E P_r + P_r S_r^{(3)}(E) P_r.
$$
Therefore we obtain 
$$
b_r = {\mathbb E}\trace (P_r S_r(E)P_r) = 
- {\mathbb E} \trace (P_r E C_r^2 E P_r)+ 
{\mathbb E} \trace (P_r S_r^{(3)}(E) P_r)
$$
$$
= - {\mathbb E}\|P_r E C_r\|_2^2 +{\mathbb E} \trace (P_r S_r^{(3)}(E) P_r)
= - \frac{1}{2}{\mathbb E} \|P_r E C_r + C_r E P_r\|_2^2 
+ {\mathbb E} \trace (P_r S_r^{(3)}(E) P_r) 
$$
$$
- \frac{1}{2}{\mathbb E}\|L_r(E)\|_2^2 + {\mathbb E} \trace (P_r S_r^{(3)}(E) P_r)
= - \frac{1}{2} \frac{A_r(\Sigma)}{n} + {\mathbb E} \trace (P_r S_r^{(3)}(E) P_r).
$$
Thus, $\beta_r = {\mathbb E} \trace (P_r S_r^{(3)}(E) P_r)$ and, using 
bound \eqref{hatSigmap}, we get
$$
|\beta_r|\leq {\mathbb E}\|S_r^{(3)}(E)\| \|P_r\|_1
\leq  {\mathbb E}\|S_r^{(3)}(E)\|
\lesssim \frac{{\mathbb E}\|E\|^3}{g_r^3}
$$
$$
\lesssim \frac{\|\Sigma\|^3}{g_r^3} \biggl(\sqrt{\frac{{\bf r}(\Sigma)}{n}} \bigvee 
\frac{{\bf r}(\Sigma)}{n}\biggr)^3,
$$
which completes the proof.
\end{proof}

It follows from the lower bound \eqref{bd_Ar_down} on $\frac{A_r(\Sigma)}{2}$ and the bound of Lemma 
\ref{brbou} that, under the assumption ${\bf r}(\Sigma)\leq n,$ with some constant $C>0$
\begin{equation}
\label{b_r_lower}
|b_r| \geq \frac{\mu_1\mu_r}{(\mu_1-\mu_r)^2 \vee \mu_r^2} 
\frac{{\bf r}(\Sigma)-1}{n}- C\frac{\|\Sigma\|^3 }{g_r^3} \biggl(\frac{{\bf r}(\Sigma)}{n}\biggr)^{3/2}.
\end{equation}
Next note that 
$$
|\langle \hat \theta_r -\theta_r, u\rangle| 
\geq |\sqrt{1+b_r}-1|\langle \theta_r,u\rangle|
- |\langle \hat \theta_r -\sqrt{1+b_r}\theta_r, u\rangle|
$$
$$
\geq 
\frac{|b_r|}{1+\sqrt{1+b_r}}|\langle \theta_r,u\rangle|
- |\langle \hat \theta_r -\sqrt{1+b_r}\theta_r, u\rangle|
$$
$$
\geq 
\frac{|b_r|}{2}|\langle \theta_r,u\rangle|
- |\langle \hat \theta_r -\sqrt{1+b_r}\theta_r, u\rangle|.
$$
Using bounds \eqref{on_hat_theta} and \eqref{b_r_lower}, we obtain
that for all $t\geq 1$ with probability at least $1-e^{-t}$
\begin{equation}
\label{chi-chi}
|\langle \hat \theta_r -\theta_r, u\rangle| \geq 
\frac{1}{2}|\langle \theta_r,u\rangle|
\biggl(\frac{\mu_1\mu_r}{(\mu_1-\mu_r)^2 \vee \mu_r^2} 
\frac{{\bf r}(\Sigma)-1}{n}- C\frac{\|\Sigma\|^3 }{g_r^3} 
\biggl(\frac{{\bf r}(\Sigma)}{n}\biggr)^{3/2}\biggr)- 
C_{\gamma}  \frac{\|\Sigma\|}{g_r}\sqrt{\frac{t}{n}}\|u\|.
\end{equation}
We will show that there exists a covariance
$\Sigma_0\in {\mathcal S}^{(r)}({\frak r}_n, a, \sigma_0, u)$ such that 
$|\langle \theta_r(\Sigma_0),u\rangle|\geq \frac{\|u\|}{2},$
$$
\frac{\mu_1(\Sigma_0)\mu_r(\Sigma_0)}{(\mu_1(\Sigma_0)-\mu_r(\Sigma_0))^2 \vee \mu_r^2(\Sigma_0)}\geq c_1
$$
for some constant $c_1>0$ that might depend on $r, a, \sigma_0$
and ${\bf r}(\Sigma_0)-1 \geq \frak{r}_n/2.$ Assuming that such a $\Sigma_0$
exists, we choose $t_n\to \infty,$
$t_n= o(\frac{{\frak r}_n^2}{n})$ and applying bound 
\eqref{chi-chi} to $\Sigma=\Sigma_0,$
we immediately obtain that 
$$
\sup_{\Sigma\in {\mathcal S}^{(r)}({\frak r}_n, a, \sigma_0, u)}
{\mathbb P}_{\Sigma}\bigl\{|\langle \hat \theta_r -\theta_r(\Sigma), u\rangle| \geq
\biggl(\frac{c_1}{8}\frac{{\frak r}_n}{n}-\frac{C}{4}a^3\biggl(\frac{{\frak r}_n}{n}\biggr)^{3/2}-C_{\gamma} a\sqrt{\frac{t_n}{n}}\biggr)\|u\|\bigr\}
$$
$$
\geq 1-e^{-t_n}\to 1.
$$
Since 
$$
\biggl(\frac{c_1}{8}\frac{{\frak r}_n}{n}-\frac{C}{4}a^3\biggl(\frac{{\frak r}_n}{n}\biggr)^{3/2}-C_{\gamma} a\sqrt{\frac{t_n}{n}}\biggr)\|u\|
= \Bigl(\frac{c_1}{8}+o(1)\Bigr)\frac{{\frak r}_n}{n} \|u\|,
$$
this implies the claim of Proposition \ref{no_efficiency}.

It remains to define a $\Sigma_0$ with the desired properties. Let  
$$
\Sigma_0 = \sum_{s=1}^{r+1} \mu_s P_s,
$$
where $P_s= \theta_s \otimes \theta_s, s=1,\dots  r,$ 
$\theta_1,\dots ,\theta_r$ being arbitrary orthonormal 
vectors in ${\mathbb H}$ and $P_{r+1}$ is an orthogonal 
projection on a $d$-dimensional subspace of ${\mathbb H}$
orthogonal to $\theta_1,\dots, \theta_r.$ Let 
$\mu_s:= \mu_1 \bigl(1-\frac{s-1}{a}\bigr), s=1,\dots, r+1.$
Then $\bar g_r(\Sigma_0)=\frac{\mu_1}{a}$ and the condition 
$\frac{\|\Sigma_0\|}{\bar g_r(\Sigma_0)}\leq a$ is satisfied.
For simplicity, assume that $\|u\|=1.$ Moreover, since 
$\theta_1,\dots, \theta_r$ are arbitrary orthonormal vectors,
we can assume without loss of generality that, for $r>1,$ 
$u:= \frac{1}{\sqrt{2}}\theta_1 + \frac{1}{\sqrt{2}}\theta_r.$
Then $\langle \theta_r(\Sigma_0),u\rangle = \frac{1}{\sqrt{2}}>\frac{1}{2}\|u\|$
and, by a simple computation,
$$
\sigma_r^2(\Sigma_0;u)= \sum_{s\neq r} \frac{\mu_r \mu_s}{(\mu_r-\mu_s)^2}\|P_s u\|^2 = \frac{1}{2}\frac{\mu_1 \mu_r}{(\mu_1-\mu_r)^2} = 
\frac{1}{2}\biggl[\frac{a^2}{(r-1)^2} -
\frac{a}{r-1}\biggr].
$$
Assuming that $\sigma_0^2 \leq \frac{1}{2}\biggl[\frac{a^2}{(r-1)^2} -
\frac{a}{r-1}\biggr],$ we conclude that the condition $\sigma_r^2(\Sigma_0;u)\geq \sigma_0^2$ 
is satisfied. For $r=1,$ we can assume that $u:= \frac{1}{\sqrt{2}}\theta_1 + \frac{1}{\sqrt{2}}\theta_2$ with a slight modification of the argument.
Finally, we take dimension $d=d_n$ so that 
$$
{\bf r}(\Sigma_0) = \sum_{s=1}^r \frac{\mu_s}{\mu_1}
+\frac{\mu_{r+1}}{\mu_1}d_n = 
\sum_{s=1}^r \bigl(1-\frac{s-1}{a}\bigr) + \bigl(1-\frac{r}{a}\bigr)d_n
\in  (\frak{r}_n/2+1, \frak{r}_n].
$$
Then $\Sigma_0\in {\mathcal S}^{(r)}({\frak r}_n, a, \sigma_0, u).$
This completes the proof.

\section{Proof of Theorem \ref{efficient_tilde}}
\label{Sec:efficient_tilde}

Recall that the estimator $\check \theta_r$ is based on empirical eigenvectors 
$\hat \theta_r^{\delta_j,j}, j=1,2,3$ with parameters $\delta_j = \tau \|\hat \Sigma^{(j)}\|$
and with a proper choice of $\tau$ (as in Lemma \ref{hatPdelta}). These eigenvectors are 
in turn defined in terms of empirical spectral projections  $\hat P_r^{\delta_j,j}$ of sample covariances $\hat \Sigma^{(j)}$ (based on $\delta_j$-clusters of its spectrum $\sigma(\hat \Sigma^{(j)})$).
We will, however,
replace $\check \theta_r$ by the estimator $\tilde \theta_r$ defined in terms of empirical spectral projections $\hat P_r^{(j)}, j=1,2,3,$ $\hat P_r^{(j)}$ being the orthogonal projection onto direct sum of eigenspaces 
of $\hat \Sigma^{(j)}$ corresponding to its eigenvalues $\lambda_k(\hat \Sigma^{(j)}), k\in \Delta_r.$
Since ${\rm card}(\Delta_r)=m_r=1,$ $\hat P_r^{(j)}=\hat \theta_r^{(j)}\otimes \hat \theta_r^{(j)}$
and we can define 
$$
\hat d_r := 
\frac{\langle \hat \theta_r^{(1)}, \hat \theta_r^{(2)}\rangle}{\langle \hat \theta_r^{(2)}, \hat \theta_r^{(3)}\rangle^{1/2}}
$$
and 
$$
\tilde \theta_r := \frac{\hat \theta_r^{(1)}}{\hat d_r\vee (1/2)}.
$$ 
The reduction to this case is based on Lemma \ref{hatPdelta} (implying that 
$\hat P_r^{\delta_j,j}=\hat P_r^{(j)}$ with a high probability) and is straightforward (as in the proof of Theorem \ref{norm_approx}).

The rest of the proof is based on several lemmas stated and proved below. 

\begin{lemma}
Suppose that for some $\gamma \in (0,1)$ condition (\ref{cond_gamma}) holds for the sample covariance $\hat \Sigma^{(2)}$ based 
on $m$ observations:
\begin{equation}
\label{cond_gamma_m}
{\mathbb E}\|\hat \Sigma^{(2)}-\Sigma\|\leq \frac{(1-\gamma)g_r}{2}
\end{equation}
Then, for all $t\geq 1$ with probability at least $1-e^{-t}$ 
\begin{eqnarray}
\label{sqr_1}
&
\bigl|\langle \hat \theta_r^{(1)}, \hat \theta_r^{(2)}\rangle - \sqrt{1+b_r^{(n')}}\sqrt{1+b_r^{(m)}}\bigr|
\lesssim_{\gamma}
\frac{\|\Sigma\|^2}{g_r^2}\biggl(\sqrt{\frac{{\bf r}(\Sigma)}{m}}\bigvee \sqrt{\frac{t}{m}}\biggr)\sqrt{\frac{t}{m}}.
\end{eqnarray}
and with the same probability
\begin{eqnarray}
\label{sqr_2}
&
\bigl|\langle \hat \theta_r^{(2)}, \hat \theta_r^{(3)}\rangle - (1+b_r^{(m)})\bigr|
\lesssim_{\gamma}
\frac{\|\Sigma\|^2}{g_r^2}\biggl(\sqrt{\frac{{\bf r}(\Sigma)}{m}}\bigvee \sqrt{\frac{t}{m}}\biggr)\sqrt{\frac{t}{m}}.
\end{eqnarray}
\end{lemma}

\begin{proof}
Obviously, condition \eqref{cond_gamma_m} holds also for the sample covariance $\hat \Sigma^{(2)}$
(which is based on a sample of the same size $m$). Moreover, it also holds for the sample covariance 
$\hat \Sigma^{(1)}$ based on $n'\geq m$ observations since the sequence $n\mapsto {\mathbb E}\|\hat \Sigma_n-\Sigma\|$ is non-increasing (see, e.g., Lemma 2.4.5 in \cite{Wellner}).

The following representation is obvious:
\begin{eqnarray}
\label{repr_theta_1_theta_2}
&
\nonumber
\langle \hat \theta_r^{(1)}, \hat \theta_r^{(2)}\rangle
= \sqrt{1+b_r^{(n')}}\sqrt{1+b_r^{(m)}} \langle \theta_r,\theta_r\rangle
\\
&
\nonumber
+ 
\sqrt{1+b_r^{(m)}}\langle \hat \theta_r^{(1)}- \sqrt{1+b_r^{(n')}} \theta_r, \theta_r\rangle
\\
&
\nonumber
+ 
\sqrt{1+b_r^{(n')}}\langle \hat \theta_r^{(2)}- \sqrt{1+b_r^{(m)}} \theta_r, \theta_r\rangle
\\
&
\langle \hat \theta_r^{(1)}- \sqrt{1+b_r^{(n')}} \theta_r, \hat \theta_r^{(2)}- \sqrt{1+b_r^{(m)}} \theta_r\rangle.
\end{eqnarray}
By bound (\ref{theta_theta_r}), with probability at least $1-e^{-t}$
\begin{eqnarray}
\label{theta_theta_r_1}
&
|\langle \hat \theta_r^{(1)}- \sqrt{1+b_r^{(n')}}\theta_r,\theta_r\rangle|
\lesssim_{\gamma} \frac{\|\Sigma\|^2}{g_r^2}\biggl(\sqrt{\frac{{\bf r}(\Sigma)}{n'}}\bigvee \sqrt{\frac{t}{n'}}\biggr)\sqrt{\frac{t}{n'}}
\end{eqnarray}
Similarly, with probability at least $1-e^{-t}$
\begin{eqnarray}
\label{theta_theta_r_2}
&
|\langle \hat \theta_r^{(2)}- \sqrt{1+b_r^{(m)}}\theta_r,\theta_r\rangle|
\lesssim_{\gamma} \frac{\|\Sigma\|^2}{g_r^2}\biggl(\sqrt{\frac{{\bf r}(\Sigma)}{m}}\bigvee \sqrt{\frac{t}{m}}\biggr)\sqrt{\frac{t}{m}}.
\end{eqnarray}
To bound the last term in the right hand side of (\ref{repr_theta_1_theta_2}), we apply bound (\ref{on_hat_theta}) to $\hat \theta_r^{(1)}$ conditionally 
on the second sample (similarly to the proof of Theorem 6 in \cite{KoltchinskiiLouniciPCAAHP}). This yields that with probability at least $1-e^{-t}$
\begin{equation}
\label{theta_r^1-theta_r^2}
|\langle \hat \theta_r^{(1)}- \sqrt{1+b_r^{(n')}} \theta_r, \hat \theta_r^{(2)}- \sqrt{1+b_r^{(m)}} \theta_r\rangle|
\lesssim_{\gamma} \frac{\|\Sigma\|}{g_r}\sqrt{\frac{t}{n'}}\|\hat \theta_r^{(2)}- \sqrt{1+b_r^{(m)}} \theta_r\|.
\end{equation}
On the other hand, under the assumption that $\langle \hat \theta_r,\theta_r\rangle\geq 0,$
\begin{eqnarray}
&
\nonumber
\|\hat \theta_r^{(2)}- \sqrt{1+b_r^{(m)}} \theta_r\|\leq 
\|\hat \theta_r^{(2)}-\theta_r\|+ \Bigl|\sqrt{1+b_r^{(m)}}-1\Bigr|
\\
&
\nonumber
=\sqrt{2-2\langle \hat \theta_r^{(2)},\theta_r\rangle} + \frac{|b_r^{(m)}|}{\sqrt{1+b_r^{(m)}}+1}
\leq 
\sqrt{2-2\langle \hat \theta_r^{(2)},\theta_r\rangle^2} + |b_r^{(m)}|
\\
&
\nonumber
=\sqrt{2-2\langle \hat P_r^{(2)},P_r\rangle} + |b_r^{(m)}|
=\|\hat P_r^{(2)}-P_r\|_2 + |b_r^{(m)}|.
\\
&
\nonumber
\leq \sqrt{2}\|\hat P_r^{(2)}-P_r\| + |b_r^{(m)}|.
\end{eqnarray}
By a standard perturbation bound (see, e.g., \cite{KoltchinskiiLouniciPCAAHP}),
$$
\|\hat P_r^{(2)}-P_r\|\leq 4\frac{\|\hat \Sigma^{(2)}-\Sigma\|}{g_r}.
$$
Thus, 
\begin{equation}
\|\hat \theta_r^{(2)}- \sqrt{1+b_r^{(m)}} \theta_r\|
\leq 4\sqrt{2}\frac{\|\hat \Sigma^{(2)}-\Sigma\|}{g_r}
+|b_r^{(m)}|.
\end{equation}
Using the exponential bound \eqref{exp_bd_Sigma} on $\|\hat \Sigma^{(2)}-\Sigma\|$ 
and bound (\ref{bd_b_r}), we obtain that with probability at least
$1-e^{-t}$
\begin{equation}
\label{theta_r^2}
\bigl\|\hat \theta_r^{(2)}- \sqrt{1+b_r^{(m)}} \theta_r\bigr\|\lesssim 
\frac{\|\Sigma\|}{g_r} \biggl(\sqrt{\frac{{\bf r}(\Sigma)}{m}}\bigvee \frac{{\bf r}(\Sigma)}{m} \bigvee \sqrt{\frac{t}{m}}\bigvee \frac{t}{m}\biggr) + \frac{\|\Sigma\|^2}{g_r^2} \frac{{\bf r}(\Sigma)}{m}.
\end{equation}
Under assumption \eqref{cond_gamma_m},
we have $\frac{\|\Sigma\|}{g_r} \sqrt{\frac{{\bf r}(\Sigma)}{m}}\lesssim 1,$ which implies $\frac{\|\Sigma\|^2}{g_r^2} \frac{{\bf r}(\Sigma)}{m}\lesssim \frac{\|\Sigma\|}{g_r} \sqrt{\frac{{\bf r}(\Sigma)}{m}}.$
Thus, the first term in the right hand side of bound (\ref{theta_r^2}) is dominant. Moreover, 
we can drop the term $\frac{{\bf r}(\Sigma)}{m}$ and, for $t\leq m,$ we can also drop the term $\frac{\|\Sigma\|}{g_r}\frac{t}{m}$ in the right hand side. 
Since the left hand side of (\ref{theta_r^2}) is not larger than $2,$ for $t>m,$ the term $\frac{\|\Sigma\|}{g_r}\sqrt{\frac{t}{m}}$ is larger (up to a constant) than the left 
hand side. Thus, the term $\frac{\|\Sigma\|}{g_r}\frac{t}{m}$ can be dropped for all the values of $t$ and the bound (\ref{theta_r^2}) simplifies as follows
\begin{equation}
\label{theta_r^2''}
\Bigl\|\hat \theta_r^{(2)}- \sqrt{1+b_r^{(m)}} \theta_r\Bigr\|\lesssim 
\frac{\|\Sigma\|}{g_r} \biggl(\sqrt{\frac{{\bf r}(\Sigma)}{m}} \bigvee \sqrt{\frac{t}{m}}\biggr)
\end{equation}
and it still holds with probability at least $1-e^{-t}.$ It follows from bound (\ref{theta_r^1-theta_r^2}) and (\ref{theta_r^2''})
that for all $t\geq 1$ with probability at least $1-2 e^{-t}$ 
\begin{equation}
\label{theta_r^1-theta_r^2_NN}
|\langle \hat \theta_r^{(1)}- \sqrt{1+b_r^{(n')}} \theta_r, \hat \theta_r^{(2)}- \sqrt{1+b_r^{(m)}} \theta_r\rangle|
\lesssim_{\gamma} \frac{\|\Sigma\|^2}{g_r^2}\biggl(\sqrt{\frac{{\bf r}(\Sigma)}{m}} \bigvee \sqrt{\frac{t}{m}}\biggr)\sqrt{\frac{t}{n'}}.
\end{equation}

Taking into account that $n'\geq m,$ it easily follows from representation (\ref{repr_theta_1_theta_2}) and bounds (\ref{theta_theta_r_1}), (\ref{theta_theta_r_2}) and (\ref{theta_r^1-theta_r^2_NN}) that with probability 
at least $1-e^{-t}$
\begin{eqnarray*}
&
\nonumber
\Bigl|\langle \hat \theta_r^{(1)}, \hat \theta_r^{(2)}\rangle - \sqrt{1+b_r^{(n')}}\sqrt{1+b_r^{(m)}}\Bigr|
\lesssim_{\gamma}
\frac{\|\Sigma\|^2}{g_r^2}\biggl(\sqrt{\frac{{\bf r}(\Sigma)}{m}}\bigvee \sqrt{\frac{t}{m}}\biggr)\sqrt{\frac{t}{m}},
\end{eqnarray*}
which proves (\ref{sqr_1}). The proof of bound (\ref{sqr_2}) is similar.
\end{proof}

Define 
$$
\Delta_1 := \frac{\langle \hat \theta_r^{(1)}, \hat \theta_r^{(2)}\rangle}{\sqrt{1+b_r^{(n')}}\sqrt{1+b_r^{(m)}}}-1
$$
and 
$$
\Delta_2 := \frac{\langle \hat \theta_r^{(2)}, \hat \theta_r^{(3)}\rangle}{1+b_r^{(m)}}-1.
$$
Assuming that 
\begin{equation}
\label{cond_b_r_b_r}
1+b_r^{(n')}\geq (3/4)^2\ \ {\rm and}\ \ 1+b_r^{(m)}\geq (3/4)^2,
\end{equation}
we obtain that, for some constant $C_{\gamma}>0$ and for $t\geq 1$ on an event $E$ of probability at least $1-e^{-t}$
\begin{equation}
\label{bd_Delta_1_2}
|\Delta_1|\vee |\Delta_2| \leq C_{\gamma}\frac{\|\Sigma\|^2}{g_r^2}\biggl(\sqrt{\frac{{\bf r}(\Sigma)}{m}}\bigvee \sqrt{\frac{t}{m}}\biggr)\sqrt{\frac{t}{m}}.
\end{equation}
Next we have 
$$
\hat d_r = 
\frac{\langle \hat \theta_r^{(1)}, \hat \theta_r^{(2)}\rangle}{\langle \hat \theta_r^{(2)}, \hat \theta_r^{(3)}\rangle^{1/2}}=
\frac{\langle \hat \theta_r^{(1)}, \hat \theta_r^{(2)}\rangle/((1+b_r^{(n')})^{1/2}(1+b_r^{(m)})^{1/2})}{\langle \hat \theta_r^{(2)}, \hat \theta_r^{(3)}\rangle^{1/2}/(1+b_r^{(m)})^{1/2}}
\sqrt{1+b_r^{(n')}}
$$
$$
=\frac{1+\Delta_1}{\sqrt{1+\Delta_2}}\sqrt{1+b_r^{(n')}}=\sqrt{1+b_r^{(n')}}+ \frac{1+\Delta_1-\sqrt{1+\Delta_2}}{\sqrt{1+\Delta_2}}\sqrt{1+b_r^{(n')}},
$$
which implies 
\begin{eqnarray}
\label{hat_d_r-b_r}
&
\nonumber
\nonumber
\Bigl|\hat d_r - \sqrt{1+b_r^{(n')}}\Bigr| \leq \sqrt{1+b_r^{(n')}}\frac{\Bigl|(1+\Delta_1)^2-(1+\Delta_2)\Bigr|}{\sqrt{1+\Delta_2}(1+\Delta_1+\sqrt{1+\Delta_2})}
\\
&
\leq \frac{2|\Delta_1|+ \Delta_1^2 + |\Delta_2|}{\sqrt{1+\Delta_2}(1+\Delta_1+\sqrt{1+\Delta_2})}.
\end{eqnarray}
Under the assumption that 
\begin{equation}
\label{assume_smaller_c}
\frac{\|\Sigma\|^2}{g_r^2}\biggl(\sqrt{\frac{{\bf r}(\Sigma)}{m}}\bigvee \sqrt{\frac{t}{m}}\biggr)\sqrt{\frac{t}{m}}\leq c_{\gamma}
\end{equation}
for a sufficiently small constant $c_{\gamma}>0,$ bounds (\ref{hat_d_r-b_r})  and (\ref{bd_Delta_1_2}) imply that 
on the event $E$ 
\begin{equation}
\label{d_r/b_r}
\biggl|\frac{\hat d_r}{\sqrt{1+b_r^{(n')}}}-1\biggr| \lesssim_{\gamma} \frac{\|\Sigma\|^2}{g_r^2}\biggl(\sqrt{\frac{{\bf r}(\Sigma)}{m}}\bigvee \sqrt{\frac{t}{m}}\biggr)\sqrt{\frac{t}{m}}.
\end{equation}
Moreover, on the same event $E,$ 
\begin{eqnarray}
\label{d_r_lower}
&
\nonumber
\hat d_r \geq \sqrt{1+b_r^{(n')}}- \frac{2|\Delta_1|+ \Delta_1^2 + |\Delta_2|}{\sqrt{1+\Delta_2}(1+\Delta_1+\sqrt{1+\Delta_2})}
\\
&
\geq \frac{3}{4}- \frac{2|\Delta_1|+ \Delta_1^2 + |\Delta_2|}{\sqrt{1+\Delta_2}(1+\Delta_1+\sqrt{1+\Delta_2})}\geq \frac{1}{2},
\end{eqnarray}
\begin{equation}
\label{b_r/d_r}
\biggl|\frac{\sqrt{1+b_r^{(n')}}}{\hat d_r}-1\biggr| \lesssim_{\gamma} \frac{\|\Sigma\|^2}{g_r^2}\biggl(\sqrt{\frac{{\bf r}(\Sigma)}{m}}\bigvee \sqrt{\frac{t}{m}}\biggr)\sqrt{\frac{t}{m}}
\end{equation}
and also, using bound (\ref{bd_b_r}), we obtain that 
\begin{eqnarray}
\label{d_r-1}
&
\nonumber
|\hat d_r-1|\leq |\sqrt{1+b_r^{(n')}}-1| + \frac{2|\Delta_1|+ \Delta_1^2 + |\Delta_2|}{\sqrt{1+\Delta_2}(1+\Delta_1+\sqrt{1+\Delta_2})}
\\
&
\nonumber
\leq |b_r^{(n')}|+ \frac{2|\Delta_1|+ \Delta_1^2 + |\Delta_2|}{\sqrt{1+\Delta_2}(1+\Delta_1+\sqrt{1+\Delta_2})}
\\
&
\lesssim_{\gamma} \frac{\|\Sigma\|^2}{g_r^2} \frac{{\bf r}(\Sigma)}{n'}+ \frac{\|\Sigma\|^2}{g_r^2}\biggl(\sqrt{\frac{{\bf r}(\Sigma)}{m}}\bigvee \sqrt{\frac{t}{m}}\biggr)\sqrt{\frac{t}{m}}.
\end{eqnarray}
and 
\begin{equation}
\label{1/d_r-1}
\bigl|\frac{1}{\hat d_r}-1\bigr|
\lesssim_{\gamma} \frac{\|\Sigma\|^2}{g_r^2} \frac{{\bf r}(\Sigma)}{n'}+ \frac{\|\Sigma\|^2}{g_r^2}\biggl(\sqrt{\frac{{\bf r}(\Sigma)}{m}}\bigvee \sqrt{\frac{t}{m}}\biggr)\sqrt{\frac{t}{m}}.
\end{equation}

The key ingredient of the proof of Theorem \ref{efficient_tilde} is the following lemma.

\begin{lemma}
\label{lemma:tilde_theta}
Suppose that, for some $\gamma\in (0,1),$ conditions (\ref{cond_gamma_m})
and (\ref{cond_b_r_b_r}) hold. Then, for all $t\geq 1$ with probability at least 
$1-e^{-t}$
\begin{eqnarray}
\label{tilde_theta-L_r}
&
\nonumber
\bigl|\langle \tilde \theta_r-\theta_r,u\rangle-\langle L_r(\hat \Sigma^{(1)}-\Sigma)\theta_r,u\rangle\bigr|
\\
&
\lesssim_{\gamma}
\frac{\|\Sigma\|^2}{g_r^2}\biggl(\sqrt{\frac{{\bf r}(\Sigma)}{m}}\bigvee \sqrt{\frac{t}{m}}\bigvee \frac{t}{m}\biggr)\sqrt{\frac{t}{m}}\|u\|.
\end{eqnarray}
\end{lemma}

\begin{proof}
We use the following simple representation:
\begin{eqnarray}
&
\nonumber
\langle \tilde \theta_r-\theta_r,u\rangle = \langle \hat \theta_r^{(1)} - \sqrt{1+b_r^{(n')}}\theta_r,u\rangle
\\
&
 + \biggl(\frac{1}{\hat d_r}-1\biggr)\langle \hat \theta_r^{(1)} - \sqrt{1+b_r^{(n')}}\theta_r,u\rangle
+ \biggl(\frac{\sqrt{1+b_r^{(n')}}}{\bar d_r}-1\biggr) \langle \theta_r,u\rangle
\end{eqnarray}
that holds on the event $E$ (where $\hat d_r\geq 1/2$).
Using bounds (\ref{b_r/d_r}) and (\ref{1/d_r-1}) that both hold under assumption (\ref{assume_smaller_c}) on the event $E$
as well as bound (\ref{on_hat_theta}) (applied to $\hat \theta_r^{(1)}$ with $n=n'$), we obtain that with probability at least $1-2e^{-t}$
\begin{eqnarray*}
&
\nonumber
\biggl|\langle \tilde \theta_r-\theta_r,u\rangle-\langle \hat \theta_r^{(1)} - \sqrt{1+b_r^{(n')}}\theta_r,u\rangle\biggr|
\\
&
\nonumber
\lesssim_{\gamma}
\frac{\|\Sigma\|^2}{g_r^2} \frac{{\bf r}(\Sigma)}{n'} \frac{\|\Sigma\|}{g_r}\sqrt{\frac{t}{n'}}\|u\|
+ \frac{\|\Sigma\|^2}{g_r^2}\biggl(\sqrt{\frac{{\bf r}(\Sigma)}{m}}\bigvee \sqrt{\frac{t}{m}}\biggr)\sqrt{\frac{t}{m}}
\frac{\|\Sigma\|}{g_r}\sqrt{\frac{t}{n'}}\|u\|
\\
&
\nonumber
+\frac{\|\Sigma\|^2}{g_r^2}\biggl(\sqrt{\frac{{\bf r}(\Sigma)}{m}}\bigvee \sqrt{\frac{t}{m}}\biggr)\sqrt{\frac{t}{m}}\|u\|.
\end{eqnarray*}
It is easy to check that the last term in the right hand side is dominant yielding the simpler bound 
\begin{eqnarray}
\label{tilde_theta-theta}
&
\nonumber
\bigl|\langle \tilde \theta_r-\theta_r,u\rangle-\langle \hat \theta_r^{(1)} - \sqrt{1+b_r^{(n')}}\theta_r,u\rangle\bigr|
\\
&
\lesssim_{\gamma}
\frac{\|\Sigma\|^2}{g_r^2}\biggl(\sqrt{\frac{{\bf r}(\Sigma)}{m}}\bigvee \sqrt{\frac{t}{m}}\biggr)\sqrt{\frac{t}{m}}\|u\|
\end{eqnarray}
that holds under assumption (\ref{assume_smaller_c}) with probability at least $1-e^{-t}.$
Since the left hand side is bounded by $5\|u\|,$ bound (\ref{tilde_theta-theta}) also holds trivially when 
$$
\frac{\|\Sigma\|^2}{g_r^2}\biggl(\sqrt{\frac{{\bf r}(\Sigma)}{m}}\bigvee \sqrt{\frac{t}{m}}\biggr)\sqrt{\frac{t}{m}}> c_{\gamma}.
$$
It remains to combine (\ref{tilde_theta-theta}) with the bound (\ref{theta-L}) (applied to $\hat \theta_r^{(1)}$) to complete the proof.

\end{proof}

The following statement is an immediate consequence of Lemma \ref{lemma:tilde_theta} and Lemma \ref{Lemma:conc_L_r}. As always, we dropped the terms $\frac{t}{n'},$ $\frac{t}{m}$ from the bounds since
the left-hand side is smaller that $3\|u\|$ and, for $t\geq n'$ or $t\geq m$ (the only cases when these terms might be needed), it is dominated by the expression with $\sqrt{\frac{t}{n'}},$ $\sqrt{\frac{t}{m}}$
only.

\begin{Corollary}
\label{tilde_theta:exponent}
Suppose that, for some $\gamma\in (0,1),$ conditions (\ref{cond_gamma_m})
and (\ref{cond_b_r_b_r}) hold. Then, for all $t\geq 1$ with probability at least 
$1-e^{-t}$
\begin{eqnarray}
\label{tilde_theta_exponent}
&
\bigl|\langle \tilde \theta_r-\theta_r,u\rangle\bigr|
\lesssim_{\gamma}
\frac{\|\Sigma\|}{g_r}\sqrt{\frac{t}{n'}}\|u\|+
\frac{\|\Sigma\|^2}{g_r^2}\biggl(\sqrt{\frac{{\bf r}(\Sigma)}{m}}\bigvee \sqrt{\frac{t}{m}}\biggr)\sqrt{\frac{t}{m}}\|u\|.
\end{eqnarray}
\end{Corollary}

Lemma \ref{lemma:tilde_theta} implies the following statement. 
This, in turn, implies Theorem \ref{efficient_tilde}.

\begin{lemma}
\label{th:normal_tilde}
Suppose that $m^2 \geq 2n{\bf r}(\Sigma)$ and conditions (\ref{cond_gamma_m}) and (\ref{cond_b_r_b_r}) hold for some $\gamma\in (0,1).$
For a given $u\in {\mathbb H},$ suppose that $\sigma_r(\Sigma;u)>0.$ Let $\alpha\geq 1.$ Then the following bounds holds: 
for some constants $C,C_{\gamma, \alpha}>0,$
\begin{align}
\label{norm_approx_A_tilde}
&
\nonumber
\sup_{x\in {\mathbb R}}\bigl|{\mathbb P}\bigl\{\frac{\sqrt{n}\langle \tilde \theta_r- \theta_r,u\rangle}{\sigma_r(\Sigma;u)}\leq x\bigr\}-\Phi(x)\bigr|
\\
&
\leq  C ({n'})^{-1/2} + \frac{C_{\gamma,\alpha}}{\sigma_r(\Sigma;u)}
\frac{\|\Sigma\|^2}{g_r^2}\biggl(\sqrt{\frac{n{\bf r}(\Sigma)}{m^2} \log \frac{m^2}{n{\bf r}(\Sigma)}}\bigvee \sqrt{\frac{n\log^2 \frac{m^2}{n{\bf r}(\Sigma)}}{m^2}}\biggr)\|u\|+ \biggl(\frac{n{\bf r}(\Sigma)}{m^2}\biggr)^{\alpha}.
\end{align}
Moreover, denote
$$
\tau_1 := C_{\gamma}\biggl(\frac{\|\Sigma\|}{g_r}\bigvee \frac{\|\Sigma\|^2}{g_r^2}\sqrt{\frac{n{\bf r}(\Sigma)}{m^2}}\biggr)\|u\|
$$
and 
$$
\tau_2:= C_{\gamma}\frac{\|\Sigma\|^2}{g_r^2}\sqrt{\frac{n}{m^2}}\|u\|.
$$
Suppose that Assumptions \ref{assump_loss} on the loss $\ell$ holds and 
$c_2\tau_2 \leq 1/4.$ 
There exist constants $C, C_{\gamma}, C_{\gamma,\alpha}>0$ such that 
\begin{align}
\label{bd_loss_tilde_A}
&
\nonumber
\bigl|{\mathbb E} \ell\biggl( \frac{\sqrt{n}\langle \tilde \theta_r- \theta_r,u\rangle}{\sigma_r(\Sigma;u)}\biggr)-{\mathbb E}\ell(Z)\bigr|
\\
&
\nonumber
\leq c_1 e^{c_2 A}
 \biggl(C({n'})^{-1/2} + 
 \frac{C_{\gamma,\alpha}}{\sigma_r(\Sigma;u)}
\frac{\|\Sigma\|^2}{g_r^2}\biggl(\sqrt{\frac{n{\bf r}(\Sigma)}{m^2} \log \frac{m^2}{n{\bf r}(\Sigma)}}\bigvee \sqrt{\frac{n\log^2 \frac{m^2}{n{\bf r}(\Sigma)}}{m^2}}\biggr)\|u\|+ \biggl(\frac{n{\bf r}(\Sigma)}{m^2}\biggr)^{\alpha}\biggr)
\\
&
+ 
2 e^{3/2} (2\pi)^{1/4} c_1 
e^{c_2^2 \tau_1^2}
(e^{-A^2/2\tau_1^2}\vee e^{-A/2\tau_2})
+ 
c_1 e^{c_2^2} e^{-A^2/4}.
\end{align}
\end{lemma}

\begin{proof}
The proof is similar to that of Lemma \ref{th:normal}.
To prove \eqref{norm_approx_A_tilde}, we apply the first bound of Lemma \ref{lemma_xi_eta}
to the random variables
$$
\xi :=\frac{\sqrt{n}\langle \tilde \theta_r- \theta_r,u\rangle}{\sigma_r(\Sigma;u)},\ \ 
\eta := \frac{\langle L_r(\hat \Sigma^{(1)}-\Sigma)\theta_r,u\rangle}{\sigma_r(\Sigma;u)}. 
$$
and use the bound of Lemma \ref{lemma:tilde_theta} with $t=\alpha \log \bigl(\frac{m^2}{n{\bf r}(\Sigma)}\bigr)$
to control $\delta (\xi,\eta).$

To prove the bound \eqref{bd_loss_tilde_A}, observe that, by bound \eqref{tilde_theta_exponent}, for all
$t\geq 1$ with probability at least $1-e^{-t}$
$$
|\xi| \leq \tau_1 \sqrt{t}\vee \tau_2 t.
$$
Under assumption $c_2 \tau_2\leq 1/4,$ bound \eqref{Bernstein-loss} implies that
$$
{\mathbb E} \ell^2 (\xi)\leq 
2e\sqrt{2\pi}c_1^2 e^{2c_2^2 \tau_1^2}+ \frac{e c_1^2}{1-2c_2 \tau_2}
\leq 4e\sqrt{2\pi}c_1^2 e^{2c_2^2 \tau_1^2}.
$$
Therefore, 
$$
{\mathbb E}\ell(\xi)I(|\xi|\geq A)\leq  {\mathbb E}^{1/2}\ell^2 (\xi) {\mathbb P}^{1/2}\{|\xi|\geq A\}\leq 
2 e^{3/2} (2\pi)^{1/4} c_1 
e^{c_2^2 \tau_1^2}
(e^{-A^2/2\tau_1^2}\vee e^{-A/2\tau_2}). 
$$
It remains to repeat the rest of the proof of the second statement of Lemma \ref{th:normal}.
\end{proof}

\section{Proof of Corollary \ref{efficient_tilde_empirical}}
\label{Sec:efficient_tilde_empirical}

The proof is based on a deterministic bound on 
$|\sigma_r^2(\tilde \Sigma;u) - \sigma_r^2(\Sigma;u)|$ for a small perturbation 
$\tilde \Sigma$ of $\Sigma$ provided by the following lemma. 

\begin{lemma} 
\label{Lemma pertubation sigma_u - sigma_u(Sigma+E)} 
Let $m_r=1.$ Denote $E:=\tilde \Sigma-\Sigma$ and suppose that $\|E\| \leq g_r/4.$ 
Then
\begin{equation}
\label{sigma_tilde_sigma_1}
|\sigma_r^2(\tilde \Sigma;u) - \sigma_r^2(\Sigma;u)| 
\lesssim \frac{\| \Sigma \|^2}{g_r^2} \frac{\|E\|}{g_r} \|u\|^2.
\end{equation}
and 
\begin{equation}
\label{sigma_tilde_sigma_2}
\biggl|\frac{\sigma_r(\tilde \Sigma;u)}{\sigma_r(\Sigma;u)}-1\biggr| 
\lesssim 
\frac{1}{\sigma_r^2(\Sigma;u)}\frac{\| \Sigma \|^2}{g_r^2} \frac{\|E\|}{g_r} \|u\|^2.
\end{equation}
\end{lemma}

\begin{proof} 
We use the Riesz representation of the spectral projector 
$P_r(\tilde \Sigma)$
$$
P_r(\tilde \Sigma) = - \frac{1}{2 \pi i} \oint_{\gamma_r}
R_{\tilde \Sigma}(\eta) d\eta,
$$
where $R_B(\eta)=(B-\eta I)^{-1}$ denotes the resolvent of operator $B$
and $\gamma_r$ is the circle in $\mathbb{C}$ with center $\mu_r$ and 
radius $g_r/2$ (and with counterclockwise orientation). 
Since $\|\tilde E\|\leq \frac{g_r}{4}$ and $m_r=1,$ it is easy to see that there 
is only one eigenvalue $\mu_r(\tilde \Sigma)$ of $\tilde \Sigma$
inside $\gamma_r$ and that ${\rm dist}(\eta; \sigma(\tilde \Sigma))\geq \frac{g_r}{4}, \eta \in \gamma_r.$ 
Note also that, for all $\eta \in \gamma_r,$ 
\begin{equation} 
\label{Bound Resolvent} 
\|R_\Sigma (\eta)\| \leq \frac{2}{g_r}, \ \ 
\|R_{\tilde \Sigma} (\eta)\| \leq \frac{4}{g_r}
\end{equation} 
and 
\begin{align}
R_{\tilde \Sigma}(\eta)-R_\Sigma(\eta) 
& 
= ( \Sigma - \eta I + E)^{-1} - (\Sigma - \eta I )^{-1} 
\notag 
\\
& 
= \left [\left (I + R_\Sigma(\eta) E \right )^{-1}-I\right] R_\Sigma(\eta).
\end{align}
It follows that, for all $\eta\in \gamma_r,$
\begin{align} 
\label{Bound Resolvent pertubation - resolvent sigma}
\|R_{\tilde \Sigma}(\eta)-R_\Sigma(\eta)\| 
\leq \frac{2}{g_r} \| \left (I + R_\Sigma(\eta) E \right)^{-1}-I\| 
\leq \frac{2}{g_r } \sum_{k=1}^\infty \| R_\Sigma(\eta)E \|^k  
\leq \frac{8 \| E \|}{g_r^2}.
\end{align}

Denote 
$
A(\Sigma):=\theta_r(\Sigma) \otimes u+u\otimes\theta_r(\Sigma),
$
$
B(\Sigma):=P_r(\Sigma) \otimes C_r(\Sigma)+ C_r(\Sigma) \otimes P_r(\Sigma)$ 
and 
$$
D(\Sigma):= B(\Sigma)A(\Sigma) = \theta_r (\Sigma)\otimes C_r(\Sigma)u + C_r(\Sigma)u\otimes \theta_r(\Sigma).
$$
We have
\begin{align}
\frac{1}{2 \pi i} \oint_{\gamma_r} R_{\Sigma}(\eta)\otimes  R_{\Sigma} (\eta) d \eta 
& 
= \sum_{s, s'=1}^{\infty} 
\frac{1}{2 \pi i} \oint_{\gamma_r} \frac{d \eta}{(\mu_s - \eta)(\mu_{s'} - \eta)} P_s \otimes P_{s'} 
\notag 
\\ 
& 
= \sum_{s \neq r} \frac{1}{\mu_r - \mu_s} (P_r \otimes P_s + P_s \otimes P_r) 
\notag 
\\
& 
= P_r(\Sigma) \otimes C_r(\Sigma)+ C_r(\Sigma) \otimes P_r(\Sigma)=B(\Sigma).
\end{align}
Hence, using \eqref{Bound Resolvent} and \eqref{Bound Resolvent pertubation - resolvent sigma}, we derive the following bound for any bounded operator $H:$
\begin{align}
\label{odin}
& \big \| (B(\tilde \Sigma)-B(\Sigma))H\big \| 
\notag  
\\ 
= 
& 
\big \|\frac{1}{2 \pi i} \oint_{\gamma_r} 
\left [R_{\tilde \Sigma}(\eta) H R_{\tilde \Sigma}(\eta)-R_{\Sigma}(\eta)H R_\Sigma (\eta)\right ]d \eta \big \| 
\notag 
\\
= 
& 
\big \| \frac{1}{2 \pi i} \oint_{\gamma_r} 
\left [ (R_{\tilde \Sigma}(\eta)-R_\Sigma(\eta)) H R_{\tilde \Sigma}(\eta)
+R_\Sigma(\eta) H (R_{\tilde \Sigma}(\eta)-R_\Sigma (\eta ))\right ]d \eta \big \| 
\notag 
\\
\leq 
&  
\frac{g_r}{2} \frac{8 \|E\|}{g_r^2}  \|H\| \biggl(\frac{4}{g_r}+ \frac{2}{g_r}\biggr)
\leq 
\frac{24 \| E \| \|H\|}{g_r^2}. 
\end{align}
Note also that 
\begin{equation}
\label{dva}
\|A(\tilde \Sigma)\|\leq 2\|u\|,
\end{equation}
and, using the bound $\|C_r(\Sigma)\|\leq \frac{1}{g_r},$ 
\begin{equation}
\label{tri}
\|B(\Sigma)H\| \leq \|P_r(\Sigma)HC_r(\Sigma)\|+ \|C_r(\Sigma)HP_r(\Sigma)\|
\leq \frac{2}{g_r}\|H\|. 
\end{equation}
Finally, observe that, by standard perturbation bounds, 
\begin{align}
\label{piat}
&
\notag
\|A(\tilde \Sigma)-A(\Sigma)\| 
\leq 2\|\theta_r(\tilde \Sigma)-\theta_r(\Sigma)\|\|u\|
\\
&
\notag
\leq 2\|P_r(\tilde \Sigma)-P_r(\Sigma)\|_2\|u\|
\leq 2\sqrt{2}\|P_r(\tilde \Sigma)-P_r(\Sigma)\|_2\|u\|
\\
&
\leq \frac{8\sqrt{2}\|E\| \|u\|}{g_r}.
\end{align}
It follows from bounds \eqref{odin}, \eqref{dva}, \eqref{tri} and \eqref{piat}
that 
\begin{align}
\label{Bound pertubation D - D}
&
\notag
\|D(\tilde \Sigma)-D(\Sigma)\| \leq 
\|(B(\tilde \Sigma)-B(\Sigma))A(\tilde \Sigma)\|+ \|B(\Sigma)(A(\tilde \Sigma)-A(\Sigma))\|
\\
&
\notag
\leq \frac{24 \| E \| \|A(\tilde \Sigma)\|}{g_r^2} +  \frac{2}{g_r}\|A(\tilde \Sigma)-A(\Sigma)\|
\leq \frac{48\|E\| \|u\|}{g_r^2}+\frac{2}{g_r}\frac{8\sqrt{2}\|E\| \|u\|}{g_r} 
\\
&
\leq  \frac{80 \|E\|\|u\|}{g_r^2}.
\end{align}
Now, recall that 
\begin{align}
\label{var_D_D}
&
\sigma_r^2(\Sigma;u)= \langle \Sigma \theta_r(\Sigma), \theta_r(\Sigma)\rangle 
\langle \Sigma C_r(\Sigma)u, C_r(\Sigma)u\rangle
\notag
\\
&
=\frac{1}{2}\Bigl\|\Sigma^{1/2}(\theta_r (\Sigma)\otimes C_r(\Sigma)u + C_r(\Sigma)u\otimes \theta_r)\Sigma^{1/2}\Bigr\|_2^2
\notag
\\
&
=\frac{1}{2}\|\Sigma^{1/2}D(\Sigma)\Sigma^{1/2}\|_2^2 
= \frac{1}{2}
{\rm tr}(\Sigma^{1/2}D(\Sigma)\Sigma^{1/2}\Sigma^{1/2}D(\Sigma)\Sigma^{1/2})
\notag
\\
&
=\frac{1}{2}
{\rm tr}(\Sigma D(\Sigma)\Sigma D(\Sigma)).
\end{align}
Hence, by the duality between operator and nuclear norms and since $\rank(D(\Sigma))\leq 2,
\rank(D(\tilde \Sigma))\leq 2,$ we have that  
\begin{align} 
\label{Bound pertubation sigma - sigma 1} 
|\sigma_r^2(\tilde \Sigma;u) - \sigma_r^2(\Sigma;u)|  
=
& 
\frac{1}{2} \big | \trace (\tilde \Sigma D(\tilde \Sigma) \tilde \Sigma D(\tilde \Sigma))- 
\trace (\Sigma D(\Sigma) \Sigma D(\Sigma)) \big | 
\notag 
\\
= 
& 
\frac{1}{2} \big | \trace((\tilde \Sigma - \Sigma)D(\tilde \Sigma) \tilde \Sigma D(\tilde \Sigma)) 
+\trace(\Sigma(D(\tilde \Sigma)-D(\Sigma))\tilde\Sigma D(\tilde \Sigma))  
\notag 
\\
&
+ \trace(\Sigma D (\Sigma) (\tilde \Sigma-\Sigma)D(\tilde \Sigma))
+ \trace (\Sigma D(\Sigma) \Sigma (D(\tilde \Sigma)-D(\Sigma))) \big | 
\notag 
\\
& 
\leq \frac{1}{2}\|\tilde \Sigma - \Sigma \| 
\big ( \|D(\tilde \Sigma) \tilde \Sigma D(\tilde \Sigma)\|_1 
+ \|D(\tilde \Sigma) \Sigma D(\Sigma) \|_1 \big ) 
\notag 
\\
&
+ 
\frac{1}{2}\|D(\tilde \Sigma) - D\| \big (\|\tilde \Sigma D(\tilde \Sigma) \Sigma\|_1  + 
\|\Sigma D(\Sigma) \Sigma\|_1   \big )
\notag 
\\
& 
\leq \|\tilde \Sigma - \Sigma \| 
\big ( \|D(\tilde \Sigma) \tilde \Sigma D(\tilde \Sigma)\| + 
\|D(\tilde \Sigma) \Sigma D(\Sigma) \| \big ) 
\notag 
\\
& 
+ 
\|D(\tilde \Sigma) - D\| \big (\|\tilde \Sigma D(\tilde \Sigma) \Sigma\|  + 
\|\Sigma D(\Sigma) \Sigma\|   \big ).
\end{align}
It remains to observe that $\|C_r(\Sigma)\|\leq \frac{1}{g_r},$
$\|C_r(\tilde \Sigma)\|\leq \frac{2}{g_r}$ and that 
$$
\|D(\Sigma)\|\leq 2 \| (\theta_r (\Sigma)\otimes C_r(\Sigma)u)\| 
\leq 2 \|C_r(\Sigma)\| \|u\|\leq \frac{2 \|u \|}{g_r},
$$
$$
\|D(\tilde \Sigma)\| \leq 2 \| (\theta_r (\tilde \Sigma)\otimes C_r(\tilde \Sigma)u)\| 
\leq 2 \|C_r(\tilde \Sigma)\| \|u\|\leq \frac{4 \|u \|}{g_r}
$$
and 
$$
\| \tilde \Sigma \| \leq \| \Sigma \| + \|E\| \leq \|\Sigma\|+ \frac{g_r}{4}\leq 
2\|\Sigma\|,
$$ 
implying the bounds 
\begin{align}
\label{DSigmaD}
&
\|D(\tilde \Sigma) \tilde \Sigma D(\tilde \Sigma)\|\leq \frac{32 \|\Sigma\| \|u\|^2}{g_r^2},
\ \  
\|D(\tilde \Sigma) \Sigma D(\Sigma)\| \leq \frac{8 \|\Sigma\| \|u\|^2}{g_r^2},
\notag 
\\
&
\|\tilde \Sigma D(\tilde \Sigma) \Sigma\| \leq \frac{8 \|\Sigma\|^2 \|u\|}{g_r}
\ \ {\rm and}\ \ 
\|\Sigma D(\Sigma) \Sigma\|  \leq \frac{2 \|\Sigma\|^2 \|u\|}{g_r}.
\end{align}
Bound \eqref{sigma_tilde_sigma_1} now follows from 
\eqref{Bound pertubation sigma - sigma 1}, 
\eqref{Bound pertubation D - D} and \eqref{DSigmaD}.
Bound \eqref{sigma_tilde_sigma_2} follows from 
\eqref{sigma_tilde_sigma_1}. 
\end{proof}

It remains to apply this lemma to $\tilde \Sigma=\hat \Sigma$ and to use 
standard bounds on $\|\hat \Sigma-\Sigma\|$ to obtain the following inequalities.

\begin{proposition}
Suppose that condition (\ref{cond_gamma}) holds for some $\gamma\in (0,1).$
Then, there exists a constant $c_{\gamma}>0$ such that for all $t\in [1,c_{\gamma}n]$ with probability at least $1-e^{-t}$
\begin{equation}
\label{sigma_tilde_sigma_1_emp}
|\sigma_r^2(\hat \Sigma;u) - \sigma_r^2(\Sigma;u)| 
\lesssim \frac{\| \Sigma \|^3}{g_r^3} \biggl(\sqrt{\frac{{\bf r}(\Sigma)}{n}}
\bigvee \sqrt{\frac{t}{n}}\biggr)\|u\|^2
\end{equation}
and 
\begin{equation}
\label{sigma_tilde_sigma_2_emp}
\biggl|\frac{\sigma_r(\hat \Sigma;u)}{\sigma_r(\Sigma;u)}-1\biggr| 
\lesssim 
\frac{1}{\sigma_r^2(\Sigma;u)}
 \frac{\| \Sigma \|^3}{g_r^3} \biggl(\sqrt{\frac{{\bf r}(\Sigma)}{n}}
\bigvee \sqrt{\frac{t}{n}}\biggr)\|u\|^2.
\end{equation}
\end{proposition}

The consistency of estimator $\sigma_r(\hat \Sigma;u)$ immediately 
follows:

\begin{proposition}
\label{consistency_sigma_r}
Suppose ${\frak r}_n>1,$ ${\frak r}_n=o(n)$ as $n\to\infty.$ For any sequence 
$\delta_n\to 0$ such that $\frac{\frak{r}_n}{n}=o(\delta_n^2)$ as $n\to\infty,$
\begin{eqnarray}
\label{consist_sigma_r}
&
\nonumber
\sup_{\Sigma \in {\mathcal S}^{(r)}({\frak r}_n, a, \sigma_0, u)}
{\mathbb P}_{\Sigma}\bigl\{\bigl|\frac{\sigma_r(\hat \Sigma;u)}{\sigma_r(\Sigma;u)}-1\bigr| \geq \delta_n\bigr\}\to 0\ {\rm as}\ n\to\infty.
\end{eqnarray}
\end{proposition}

Corollary \ref{efficient_tilde_empirical} can be easily proved using 
the first statement of Theorem \ref{efficient_tilde},
Proposition \ref{consistency_sigma_r} and Lemma \ref{lemma_xi_eta}.

\section{Proof of Theorem \ref{Thm Minimax lower bounds 1}}
\label{sec:van Trees}

Note that the set $\mathring{\mathcal S}^{(r)}({\frak r}, a, \sigma_0, u)$
is open in nuclear norm topology. This easily follows from the continuity 
of functions $\Sigma \mapsto \|\Sigma\|,$ $\Sigma \mapsto \bar g_r(\Sigma)$
and $\Sigma \mapsto \sigma_r^2(\Sigma;u)$ with respect to the operator norm (for the last function,
see Lemma \ref{Lemma pertubation sigma_u - sigma_u(Sigma+E)}) and, as 
a consequence, with respect to the nuclear norm, and of the functions $\Sigma \mapsto 
{\rm tr}(\Sigma)$ and $\Sigma\mapsto {\bf r}(\Sigma)$ with respect to the nuclear 
norm. 

Let 
$\Sigma=\sum_{s=1}^{\infty}\mu_s P_s\in \mathring{\mathcal S}^{(r)}({\frak r}, a, \sigma_0, u).$ 
Without loss of generality, assume that $\Sigma$ is of finite rank. 
Otherwise, consider $\Sigma_N:= \sum_{s=1}^N \mu_s P_s.$ Clearly, 
$$
{\bf r}(\Sigma_N)\leq {\bf r}(\Sigma)< {\frak r}
$$ 
and, for all $N>r,$ 
$$
\frac{\|\Sigma_N\|}{\bar g_r(\Sigma_N)}=\frac{\|\Sigma\|}{\bar g_r(\Sigma)}<a.
$$
Moreover, since $\|\Sigma_N-\Sigma\|\to 0$ as $N\to \infty,$
we also have that $\sigma_r^2(\Sigma_N;u)\to \sigma_r^2(\Sigma;u)$
as $N\to\infty,$ implying that  $\sigma_r^2(\Sigma_N;u)>\sigma_0^2$
for all large enough $N.$ Thus, 
$\Sigma_N\in \mathring{\mathcal S}^{(r)}({\frak r}, a, \sigma_0, u)$
for a sufficiently large $N$ and we can replace $\Sigma$ by $\Sigma_N.$
Assuming that ${\rm rank}(\Sigma)<\infty,$ let $L:= {\rm Im}(\Sigma).$
We can now restrict $\Sigma$ to an operator acting from $L$ to $L,$
which is non-singular. In what follows, all the covariance operators 
from the class $\mathring{\mathcal S}^{(r)}({\frak r}, a, \sigma_0, u)$ 
that are of interest to us will have $L$ as an image and could be viewed 
as operators from $L$ to $L.$ For simplicity, we just assume that ${\mathbb H}=L$ is a finite-dimensional space. For a fixed $\Sigma,$ consider the following parametric family of perturbations of $\Sigma:$ 
$$
\Sigma_t:=\Sigma+\frac{t H}{\sqrt{n}}, |t|\leq c,
$$
where $H$ is a self-adjoint operator and $c>0$ is a constant. 
Denote 
$$
{\mathcal S}_{\Sigma,c} := \{\Sigma_t: t\in [-c,c]\}. 
$$
Since the set 
$\mathring{\mathcal S}^{(r)}({\frak r}, a, \sigma_0, u)$ is open in nuclear norm topology,
there exists $\delta>0$ such that the condition 
\begin{equation}
\label{cond_H}
\frac{c \|H\|_1}{\sqrt{n}}<\delta,
\end{equation}
implies that ${\mathcal S}_{\Sigma,c}\subset \mathring{\mathcal S}^{(r)}({\frak r}, a, \sigma_0, u).$
Moreover, we will assume that 
\begin{equation}
\label{cond_delta_1}
\delta < \|\Sigma^{-1}\|^{-1}
\end{equation} 
and 
\begin{equation}
\label{cond_delta_2}
\delta< \frac{1}{4}\bar g_r(\Sigma).
\end{equation}
Under these assumptions and condition 
(\ref{cond_H}),  $\Sigma_t$ is a small enough perturbation of $\Sigma$
so that $\Sigma_t$ is non-singular and we can define in a standard way the one-dimensional spectral projection operator $P_t:=P_r(\Sigma_t)=\theta_t\otimes \theta_t,$ where $\theta_t = \theta_r(\Sigma_t)$ is the corresponding unit eigenvector as well as operators $C_t:=C_r(\Sigma_t)$ and 
$$
L_t(H):=L_r(\Sigma_t)(H)= P_t HC_t + C_t HP_t.
$$
It is easy to see that (for a given $c>0$ and large enough $n$ so that the perturbation 
is small) one can choose $t\mapsto \theta_t$ in such a way that $\langle \theta_t, \theta_{t'}\rangle\geq 0, t,t'\in [-c,c].$ Based on these definitions, we also define the functions  
$g(t):=\langle \theta_t, u\rangle$ and $\sigma^2(t):= \sigma_r^2(\Sigma_t;u).$
Concerning the function $g,$ we need the following lemma.

\begin{lemma} 
\label{Lemma lower bound help} 
The function $g$ is continuously differentiable in the interval $[-c,c]$
and the following statements hold:
\begin{enumerate}[i)]
\item $ ~~
g'(t)=  \frac{1}{\sqrt{n}}\langle L_t(H) \theta_t, u \rangle, t\in [-c,c].
$
\item ~~$
|g'(t)-g'(0)| \lesssim \frac{|t| \| H \|^2}{g_r^2 n} \|u \|, t\in [-c,c].
$
\end{enumerate}
\end{lemma}
	
\begin{proof} 
Let $\delta\in (-1,1).$ 
Similarly to \eqref{represent_theta_r} (see also (6.6) in \cite{KoltchinskiiLouniciPCAAHP}), 
\begin{align} 
\label{Decomposition Lemma for lower bound g}	
g(t+\delta) - g(t) 
& = \langle \theta_{t+\delta}-\theta_t, u \rangle \notag 
\\
& 
= 
\frac{\langle( P_{t+\delta}-P_t)\theta_t, u \rangle - (\sqrt{1+\langle ( P_{t+\delta} - P_t) \theta_t, \theta_t \rangle}-1)\langle \theta_t, u \rangle  }{\sqrt{1+\langle ( P_{t+\delta} - P_t) \theta_t, \theta_t \rangle}}.
\end{align}
Applying the first order perturbation expansion (similar to \eqref{first_perturb}) 
to the spectral projections $P_t, P_{t+\delta},$ 
we obtain that 
\begin{equation}
P_{t+\delta}-P_t=L_t(\delta H / \sqrt{n}) + S_t (\delta H / \sqrt{n})
\end{equation}
with the remainder term satisfying the bound
\begin{equation} 
\label{Bound Lemma lower nonlinear term I}
\| S_t (\delta H / \sqrt{n})\| \lesssim \frac{\delta^2 \|H \|^2}{ g_r^2 n}= O(\delta^2).
\end{equation}
Moreover, since $C_t \theta_t=0,$
\begin{equation}
\langle L_t(\delta H/\sqrt{n}) \theta_t, \theta_t \rangle
=\frac{1}{\sqrt{n}}\langle (P_t H C_t + C_t H P_t)\theta_t, \theta_t\rangle
=0
\end{equation}
and therefore we have that
\begin{equation} 
\label{Bound Lemma lower nonlinear term decomposition I}
|\langle (P_{t+\delta}-P_t) \theta_t, \theta_t \rangle| \lesssim \frac{\delta^2  \|H \|^2}{ g_r^2 n} = O(\delta^2).
\end{equation}
Hence, using again \eqref{Decomposition Lemma for lower bound g}, \eqref{Bound Lemma lower nonlinear term I} and \eqref{Bound Lemma lower nonlinear term decomposition I}, 
we have that
\begin{align}
\frac{g(t+\delta)-g(t)}{\delta} = \frac{1}{\sqrt{n}}\frac{\langle L_t(H)\theta_t, u \rangle }{1+O(\delta)} + O(\delta).
\end{align}
Passing to the limit as $\delta\to 0$ implies the first assertion. \\
	
We now prove the second claim. First note that
\begin{align} 
\label{Bound g'(t)-g'(0)}
|g'(t)-g'(0)| 
& 
= 
|\langle L_t(H/\sqrt{n}) \theta_t - L_0(H/\sqrt{n})\theta_0, u \rangle| 
\notag 
\\
& 
\leq | \langle ( L_t (H/\sqrt{n})-L_0(H/\sqrt{n}))\theta_t, u \rangle | 
+ |\langle L_0(H/\sqrt{n}) (\theta_t-\theta_0), u \rangle | 
\notag 
\\
& 
\leq \| L_t(H/\sqrt{n}) - L_0(H/\sqrt{n}) \| \|u \| 
+ \| L_0(H/\sqrt{n}) \| \| \theta_t - \theta_0 \| \| u \|.
\end{align}
Also,
\begin{equation}
\label{L_t_resolv}
L_t(H/\sqrt{n})= 
-\frac{1}{2 \pi i} \oint_{\gamma_r} R_{\Sigma_t}(\eta) \frac{H}{\sqrt{n}} R_{\Sigma_t}(\eta) d\eta, 
\end{equation}
where $\gamma_r$ is the circle of radius $g_r/2$ with the center at $\mu_r$
and with counterclockwise orientation.  Therefore, by a standard argument already used in the proof of Lemma \ref{Lemma pertubation sigma_u - sigma_u(Sigma+E)},	\begin{align} 
\label{L_t-L_0}
\|L_t(H/\sqrt{n}) - L_0(H/\sqrt{n})\| 
\leq 
\frac{g_r}{2} \sup_{\eta \in \gamma_r} \|R_{\Sigma_t}(\eta)-R_{\Sigma}(\eta)\| 
(\|R_{\Sigma} (\eta)\|+\|R_{\Sigma_t} (\eta)\|)\frac{\|H \|}{\sqrt{n}} .
\end{align}
By \eqref{Bound Resolvent} and \eqref{Bound Resolvent pertubation - resolvent sigma},
we have 
\begin{equation*} 
\|R_\Sigma (\eta)\| \leq \frac{2}{g_r}, \ \ 
\|R_{\Sigma_t} (\eta)\| \leq \frac{4}{g_r}
\end{equation*} 
and 
\begin{align*}
\| R_{\Sigma_t}(\eta)-R_{\Sigma}(\eta) \| \leq \frac{8}{g_r^2} \frac{|t| \|H \|}{\sqrt{n}}.
\end{align*}
Therefore, it follows from \eqref{L_t-L_0} that 
\begin{align} 
\label{L_t-L_0_A}
\|L_t(H/\sqrt{n}) - L_0(H/\sqrt{n})\| 
\leq 
\frac{24|t|\|H\|^2}{g_r^2 n}.
\end{align}
It remains to observe that 
$$
\|L_0(H)\| = \|P_r H C_r+ C_r H P_r\|\leq 
\frac{2\|H\|}{g_r}
$$ 
and 
$$
\|\theta_t-\theta_0\| \leq \|P_t - P_0\|_2 \leq \frac{4 \sqrt{2}|t|\|H\|}{g_r \sqrt{n}}
$$
(where we also used the fact that ${\rm rank}(P_t-P_0)\leq 2$ and 
$\|P_t-P_0\|_2\leq \sqrt{2}\|P_t-P_0\|$).
This implies the bound 
\begin{align}
\label{L_0_H}
\|L_0(H/\sqrt{n})\| \|\theta_t - \theta_0\| \|u\|
\leq \frac{8\sqrt{2} |t| \|H\|^2}{g_r^2 n}\|u\|.
\end{align}
The second assertion follows from the bounds \eqref{Bound g'(t)-g'(0)}, \eqref{L_t-L_0_A} and \eqref{L_0_H}.

The continuity of the derivative $g'(t)$ easily follows from the continuity 
of the functions $t\mapsto \theta_t$ and $t\mapsto L_t(H/\sqrt{n})$ 
(which could be proved using representation \eqref{L_t_resolv}).

\end{proof}

We will study the following estimation problem. Let 
$\Sigma$ be fixed and let $X_1,\dots, X_n$ be $i.i.d.$ random variables in ${\mathbb H}$ sampled from $N(0;\Sigma_t), |t|\leq c,$ $t$ being an unknown parameter.
The goal is to estimate the function $g(t)$ based on the observations $X_1,\dots, X_n. $	
We will use van Trees inequality to obtain a minimax lower bound on the risk 
of estimation of $g(t)$ with respect to quadratic loss. To this end, let $\pi$ be 
a smooth probability density on $[-1,1],$ satisfying the boundary conditions
$\pi(-1)=\pi(1)=0$ as well the condition $J_\pi:= \int_{-1}^{1} \frac{\pi'(s)^2}{\pi(s)}ds<+\infty.$ Let $\pi_c(t):= \frac{1}{c}\pi\bigl(\frac{t}{c}\bigr), t\in [-c,c]$ be a prior 
on $[-c,c].$ Then (see e.g. \cite{GillLevit}), for any estimator 
$T_n=T_n(X_1,\dots, X_n)$ of $g(t)$ the following bound holds
\begin{align} 
\label{inequality van trees}
\nonumber
\sup_{|t| \leq c} n \mathbb{E}_t (T_n-g(t))^2 
& \geq 
n \int_{-c}^c\mathbb{E}_t (T_n-g(t))^2\pi_c(t)dt
\\
&
\geq 
\frac{n \big (\int_{-c}^{c} g'(t) \pi_c(t) d t\big )^2}{\int_{-c}^c \mathbb{I}_n(t) \pi_c(t)dt+ 
J_{\pi_c}},
\end{align}
where $\mathbb{I}_n(t)= n{\mathbb I}(t)$ 
denotes the Fisher information for the model $$X_1,\dots, X_n \overset{i.i.d.}{\thicksim} N(0,\Sigma_t),$$
$t\in [-c,c].$ Let ${\mathbb I}(t):={\mathbb I}_1(t).$
It is well known that the Fisher information for the model $X\sim N(0;\Sigma)$
with non-singular covariance matrix $\Sigma$ is $\mathbb{I}(\Sigma)= \frac{1}{2}(\Sigma^{-1}\otimes \Sigma^{-1})$ (see, e.g., \cite{Eaton}). Thus,
$$
{\mathbb I}_n(t)= n {\mathbb I}(t)= 
n\bigl\langle {\mathbb I}(\Sigma_t)\frac{d\Sigma_t}{dt},
\frac{d\Sigma_t}{dt}\bigr\rangle  
=\frac{n}{2}\bigl\langle (\Sigma_t^{-1}\otimes \Sigma_t^{-1})\frac{H}{\sqrt{n}},
\frac{H}{\sqrt{n}}\bigr\rangle   
$$
$$
= \frac{1}{2} \langle \Sigma_t^{-1}H\Sigma_t^{-1}, H\rangle= 
\frac{1}{2} {\rm tr}(\Sigma_t^{-1}H\Sigma_t^{-1}H).
$$
We will now bound the numerator of the expression in the right hand side of inequality 
\eqref{inequality van trees} from below and its denominator from above. 

{\it Bound on the numerator.} We use Lemma \ref{Lemma lower bound help} to obtain that for some constant $B_1>0$ 
\begin{align}
\label{numerator_bd}
\big (\int_{-c}^{c} g'(t) \pi_c(t) dt \big )^2 
& = \big (\int_{-c}^{c} [g'(0)+(g'(t)-g'(0))] \pi(t/c) dt/c \big )^2 \notag \\  
& \geq g'(0)^2 + 2 g'(0) \int_{-c}^{c}( g'(t)-g'(0)) \pi(t/c)dt/c \notag \\
& \geq g'(0)^2 - 2 |g'(0)| \int_{-c}^{c}| g'(t)-g'(0)| \pi(t/c)dt/c \notag \\
& \geq g'(0)^2 - B_1 c |g'(0)| \int_{-1}^{1} |t| \pi(t)dt \frac{\|H\|^2}{g_r^2 n} \|u\| \notag \\
& = g'(0)^2 - B_{1} c |g'(0)| \frac{\| H \|^2}{g_r^2 n} \| u \| \notag \\
& = \frac{\langle L_r(H)\theta_r,u\rangle^2}{n} - 
|\langle L_r(H)\theta_r,u\rangle|
\frac{B_{1}c \| H \|^2}{g_r^2 n^{3/2}} \|u\|.
\end{align}

{\it Bound on the denominator.} 
First note that, by a simple computation,
\begin{equation}
\label{Jpi}
J_{\pi_c}= J_{\pi}/c^2.
\end{equation}
Then, we need to bound 
$\mathbb{I}_n(t)=\frac{1}{2} {\rm tr}(\Sigma_t^{-1}H\Sigma_t^{-1}H)$ in terms of 
$\mathbb{I}_n(0)=\frac{1}{2} {\rm tr}(\Sigma^{-1}H\Sigma^{-1}H).$ 
Assume that 
\begin{equation}
\label{cond_HH}
\frac{c\|\Sigma^{-1}H\|}{\sqrt{n}}\leq \frac{1}{2}.
\end{equation}
Arguing as in the proof of Lemma \ref{Lemma pertubation sigma_u - sigma_u(Sigma+E)}, we easily get that
\begin{align}
\Sigma_t^{-1} 
= \Sigma^{-1} + \underbrace{\left [ \left ( I+ \frac{t \Sigma^{-1}H}{\sqrt{n}}\right )^{-1} - I\right ]}_{=:D} \Sigma^{-1},
\end{align} 
where 
$$
\|D \| \leq 2 |t| \frac{\| \Sigma^{-1} H \|}{\sqrt{n}}\leq 1.
$$
Furthermore, note that
\begin{equation*}
\trace \left ( \Sigma_t^{-1} H \Sigma_t^{-1} H \right ) = 
\trace (\Sigma^{-1} H \Sigma^{-1} H ) + 2 \trace (D \Sigma^{-1} H \Sigma^{-1} H) + \trace ( D \Sigma^{-1} H D \Sigma^{-1} H).
\end{equation*}
and thus we have that
\begin{align}
\mathbb{I}_n(t)
& \leq \mathbb{I}_n(0) 
+ \|D \| \| \Sigma^{-1} H \Sigma^{-1} H \|_1 + 
\frac{\|D\Sigma^{-1}H\|_2\|H\Sigma^{-1}D\|_2}{2} 
\notag 
\\ 
& 
\leq \mathbb{I}_n(0)  + 
\left (\|D\| + \frac{\| D \|^2}{2}\right) \|\Sigma^{-1}H \|_2^2 
\leq \mathbb{I}_n(0)  + 
3 \frac{|t|\|\Sigma^{-1}H \|_2^3}{\sqrt{n}}.
\label{denominator_bd_A}
\end{align}
Using \eqref{denominator_bd_A}, we obtain the 
following bound:
\begin{align}
\label{denominator_bd}
\int_{-c}^c \mathbb{I}_n(t) \pi_c(t)dt 
& 
\leq 
\mathbb{I}_n(0) + 
3 \frac{ \| \Sigma^{-1}H \|_2^3}{\sqrt{n}} 
\int_{-c}^{c} |t| \pi(t/c)dt/c
\notag 
\\
& 
\leq 
\mathbb{I}_n(0)+  
\frac{3c\| \Sigma^{-1}H \|_2^3}{\sqrt{n}}.
\end{align}

Substituting \eqref{numerator_bd}, \eqref{denominator_bd} and \eqref{Jpi}
into van Trees inequality \eqref{inequality van trees} and taking into account 
that 
$$
{\mathbb I}_n(0) = \frac{1}{2}{\rm tr}(\Sigma^{-1}H\Sigma^{-1}H)= 
\frac{1}{2}\|\Sigma^{-1/2}H\Sigma^{-1/2}\|_2^2
$$
and 
$$
\langle L_r(H)\theta_r,u\rangle = \langle (P_r H C_r + C_r H P_r)\theta_r,u\rangle 
= \langle H\theta_r, C_r u\rangle
$$
$$
=\frac{1}{2} \langle H, \theta_r \otimes C_r u + C_r u \otimes \theta_r\rangle
= \langle \Sigma^{-1/2}H\Sigma^{-1/2}, \Sigma^{-1/2}B\Sigma^{-1/2}\rangle,
$$
where 
$$
B:= \frac{1}{2}(\Sigma \theta_r \otimes \Sigma C_r u+ \Sigma C_r u\otimes \Sigma \theta_r),
$$
we obtain that 
\begin{align} 
\label{inequality van trees''}
&
\nonumber
\sup_{|t| \leq c} n \mathbb{E}_t (T_n-g(t))^2 
\\
&
\geq 
\frac{\langle \Sigma^{-1/2}H\Sigma^{-1/2}, \Sigma^{-1/2}B\Sigma^{-1/2}\rangle^2 - 
|\langle \Sigma^{-1/2}H\Sigma^{-1/2}, \Sigma^{-1/2}B\Sigma^{-1/2}\rangle|
\frac{B_{1}c \| H \|^2}{g_r^2 \sqrt{n}} \|u\|}
{\frac{1}{2}\|\Sigma^{-1/2}H\Sigma^{-1/2}\|_2^2+  
\frac{3c\| \Sigma^{-1}H \|_2^3}{\sqrt{n}}+ 
J_{\pi}/c^2}.
\end{align}
In what follows, we set $H:=B.$ Note that with this choice of $H$
$$
2\|\Sigma^{-1/2}B\Sigma^{-1/2}\|_2^2= \frac{1}{2}
\|\Sigma^{1/2} \theta_r \otimes \Sigma^{1/2} C_r u+ \Sigma^{1/2} C_r u\otimes 
\Sigma^{1/2}\theta_r\|_2^2
$$
$$
=\frac{1}{2} \Bigl(\|\Sigma^{1/2} \theta_r \otimes \Sigma^{1/2} C_r u\|_2^2+ 
\|\Sigma^{1/2} C_r u\otimes 
\Sigma^{1/2}\theta_r\|_2^2\Bigr)= \|\Sigma^{1/2}\theta_r\|^2 \|\Sigma^{1/2}C_r u\|^2
= \sigma_r^2(\Sigma;u).
$$
Also, by a simple computation (using that ${\rm rank}(B)=2$), we have that 
\begin{equation}
\label{bounds_on_B_1}
\|B\|\leq \|B\|_2 \leq \frac{1}{\sqrt{2}}\frac{\|\Sigma\|^2}{g_r}\|u\|,
\ \ \|B\|_1 \leq \frac{\|\Sigma\|^2}{g_r}\|u\|
\end{equation}
and that
\begin{equation}
\label{bounds_on_B_2}
\|\Sigma^{-1}B\|\leq \|\Sigma^{-1}B\|_2 \leq \frac{1}{\sqrt{2}}\frac{\|\Sigma\|}{g_r}\|u\|.
\end{equation}
These bounds imply that, for any given $c>0$ and for all $n$ large enough, $H=B$
satisfies condition \eqref{cond_H} for a small enough $\delta$ 
such that 
${\mathcal S}_{\Sigma,c}\subset \mathring{\mathcal S}^{(r)}({\frak r}, a, \sigma_0, u)$
and conditions \eqref{cond_delta_1}, \eqref{cond_delta_2} hold. 
Also, $H=B$ satisfies condition \eqref{cond_HH} (for any given $c>0$ and all large enough $n$). 

For $H=B,$ inequality \eqref{inequality van trees''} becomes 
\begin{align} 
\label{inequality van trees_fin}
&
\nonumber
\sup_{|t| \leq c} n \mathbb{E}_t (T_n-g(t))^2 
\\
&
\nonumber
\geq 
\frac{\|\Sigma^{-1/2}B\Sigma^{-1/2}\|_2^4 - 
\|\Sigma^{-1/2}B\Sigma^{-1/2}\|_2^2
\frac{B_{1}c \|B\|^2}{g_r^2 \sqrt{n}} \|u\|}
{\frac{1}{2}\|\Sigma^{-1/2}B\Sigma^{-1/2}\|_2^2+  
\frac{3c\| \Sigma^{-1}B\|_2^3}{\sqrt{n}}+ 
J_{\pi}/c^2}
\\
&
\geq 
\sigma_r^2(\Sigma;u)
\biggl(1-
\frac{\frac{B_{1}c \|B\|^2}{2g_r^2 \sqrt{n}} \|u\|  + 
\frac{3c\| \Sigma^{-1}B\|_2^3}{\sqrt{n}}+ J_{\pi}/c^2}
{\frac{1}{4}\sigma_r^2(\Sigma;u)+  
\frac{3c\| \Sigma^{-1}B\|_2^3}{\sqrt{n}}+ 
J_{\pi}/c^2}\biggr).
\end{align}
It remains to replace $\sigma_r^2(\Sigma;u)$ with $\sigma^2(t)=\sigma_r^2(\Sigma_t;u).$
To this end, we use the bound \eqref{sigma_tilde_sigma_1} to obtain that for some 
constant $D_1>0$
\begin{equation}
\label{sigma_tilde_sigma_1''}
\sup_{t\in [-c,c]}\frac{\sigma^2(t)}{\sigma_r^2(\Sigma;u)}
\leq 1+\frac{D_1}{\sigma_r^2(\Sigma;u)}\frac{\| \Sigma \|^2}{g_r^3} \frac{c\|B\|}{\sqrt{n}} \|u\|^2.
\end{equation}
It follows from \eqref{inequality van trees_fin} that 
\begin{align} 
\label{inequality van trees_fin''}
&
\sup_{t\in [-c,c]}\frac{\sigma^2(t)}{\sigma_r^2(\Sigma;u)}
\sup_{|t| \leq c} \frac{n \mathbb{E}_t (T_n-g(t))^2}{\sigma^2(t)} 
\geq 
1-
\frac{\frac{B_{1}c \|B\|^2}{2g_r^2 \sqrt{n}} \|u\|  + 
\frac{3c\| \Sigma^{-1}B\|_2^3}{\sqrt{n}}+ J_{\pi}/c^2}
{\frac{1}{4}\sigma_r^2(\Sigma;u)+  
\frac{3c\| \Sigma^{-1}B\|_2^3}{\sqrt{n}}+ 
J_{\pi}/c^2}.
\end{align}
Suppose 
\begin{equation}
\label{assump_XYZ}
\frac{D_1}{\sigma_r^2(\Sigma;u)}\frac{\| \Sigma \|^4}{g_r^4} \frac{c}{\sqrt{n}} 
\|u\|^3 \leq 1,
\end{equation}
which holds for any given $c>0$ and all large enough $n$ and which, in view of bounds  \eqref{bounds_on_B_1}, implies that 
\begin{equation*}
\frac{D_1}{\sigma_r^2(\Sigma;u)}\frac{\| \Sigma \|^2}{g_r^3} \frac{c\|B\|}{\sqrt{n}} 
\|u\|^2 \leq 1.
\end{equation*}
Under condition \eqref{assump_XYZ},
bounds \eqref{inequality van trees_fin''} and \eqref{sigma_tilde_sigma_1''} 
(and also bounds \eqref{bounds_on_B_1} and \eqref{bounds_on_B_2}) imply 
that 
\begin{align} 
\label{inequality van trees_final}
&
\nonumber
\sup_{\Sigma\in \mathring{\mathcal S}^{(r)}({\frak r}, a, \sigma_0, u)}
\frac{{\mathbb E}_{\Sigma}(T_n-\langle \theta_r(\Sigma),u\rangle)^2}{\sigma_r^2(\Sigma;u)}
\geq 
\sup_{|t| \leq c} \frac{n \mathbb{E}_t (T_n-g(t))^2}{\sigma^2(t)} 
\\
&
\nonumber
\geq 
\biggl(1-
\frac{\frac{B_{1}c \|B\|^2}{2g_r^2 \sqrt{n}} \|u\|  + 
\frac{3c\| \Sigma^{-1}B\|_2^3}{\sqrt{n}}+ J_{\pi}/c^2}
{\frac{1}{4}\sigma_r^2(\Sigma;u)+  
\frac{3c\| \Sigma^{-1}B\|_2^3}{\sqrt{n}}+ 
J_{\pi}/c^2}\biggr)
\biggl(1-\frac{D_1}{\sigma_r^2(\Sigma;u)}\frac{\| \Sigma \|^2}{g_r^3} \frac{c\|B\|}{\sqrt{n}} \|u\|^2\biggr)
\\
&
\geq 
\biggl(1-
\frac{B_1 a^4 \|u\|^3 \frac{c}{\sqrt{n}}   + 
3a^3 \|u\|^3\frac{c}{\sqrt{n}}+ J_{\pi}/c^2}
{\frac{\sigma_0^2}{4}+
3a^3 \|u\|^3\frac{c}{\sqrt{n}}+ 
J_{\pi}/c^2}\biggr)
\biggl(1-\frac{D_1}{\sigma_0^2}a^4\|u\|^3\frac{c}{\sqrt{n}}\biggr).
\end{align}
It remains to pass to the limit in inequality \eqref{inequality van trees_final}
first as $n\to\infty$ and then as $c\to \infty$ to complete the proof. 

A local version of the theorem easily follows from the above arguments since, for all $\varepsilon>0, c>0$
and for all large enough $n,$
$
{\mathcal S}_{\Sigma_0,c}\subset \{\Sigma: \|\Sigma-\Sigma_0\|_1\leq \varepsilon\}.
$
\\

\textbf{Acknowledgements.} The authors thank an Associate Editor and two referees for careful reading of and critical remarks on the manuscript. The first author thanks the Department of Pure Mathematics and Mathematical Statistics of the University of Cambridge for its hospitality during visits in fall 2016 and summer 2017, partly funded by ERC grant UQMSI/647812.

\end{document}